\documentclass{article}
\usepackage{CJK,CJKnumb,CJKulem,times,dsfont,ifthen,mathrsfs,latexsym,amsfonts,color}
\usepackage{amsmath,amsthm,makeidx,fontenc,amssymb,bm,graphicx,psfrag,listings,curves,extarrows,enumitem}
\usepackage{hyperref}
\usepackage[title]{appendix}

\allowdisplaybreaks[4] 

\usepackage{geometry}
\usepackage{indentfirst}
\usepackage{url}
\makeatletter
\def\url@leostyle{%
  \@ifundefined{selectfont}{\def\UrlFont{\sf}}{\def\UrlFont{\small\ttfamily}}}
\makeatother
\usepackage{amssymb}
\makeatletter

\newcommand{\Rmnum}[1]{\expandafter\@slowromancap\romannumeral #1@}
\makeatother

\newtheorem{theorem}{Theorem}[section]   
\newtheorem{definition}[theorem]{Definition} 	 
\newtheorem{lemma}[theorem]{Lemma}	 
\newtheorem{corollary}[theorem]{Corollary}	 
\newtheorem{proposition}[theorem]{Proposition}	 
\newtheorem{remark}{Remark}[section]

\newcommand{\al}{\alpha}
\newcommand{\ga}{\gamma}

\newcommand{\e}{\varepsilon}

\newcommand{\iy}{\infty}
\newcommand{\q}{\theta}

\newcommand{\la}{\lambda}
\newcommand{\vp}{\varphi}
\newcommand{\pa}{\partial}

\newcommand{\rh}{\rightharpoonup}

\newcommand{\lab}{\label}
\newcommand{\f}{\frac}
\newcommand{\bt}{\begin{theorem}}
\newcommand{\et}{\end{theorem}}
\newcommand{\bl}{\begin{lemma}}
\newcommand{\el}{\end{lemma}}
\newcommand{\bpr}{\begin{proposition}}
\newcommand{\epr}{\end{proposition}}
\newcommand{\bd}{\begin{definition}}
\newcommand{\ed}{\end{definition}}
\newcommand{\bc}{\begin{corollary}}
\newcommand{\ec}{\end{corollary}}
\newcommand{\bp}{\begin{proof}}
\newcommand{\ep}{\end{proof}}
\newcommand{\bx}{\begin{example}}
\newcommand{\ex}{\end{example}}
\newcommand{\bi}{\begin{exercise}}
\newcommand{\ei}{\end{exercise}}
\newcommand{\br}{\begin{remark}}
\newcommand{\er}{\end{remark}}
\newcommand{\be}{\begin{equation}}
\newcommand{\ee}{\end{equation}}
\newcommand{\bal}{\begin{align}}
\newcommand{\bn}{\begin{enumerate}}
\newcommand{\en}{\end{enumerate}}
\newcommand{\ba}{\begin{align}}
\newcommand{\ea}{\begin{align}}
\newcommand{\bg}{\begin{align*}}
\newcommand{\eg}{\end{align*}}
\newcommand{\bcs}{\begin{cases}}
\newcommand{\ecs}{\end{cases}}

\newcommand{\CR}{{\cal C}}

\newcommand{\HR}{{\cal H}}

\newcommand{\MR}{{\cal M}}
\newcommand{\NR}{{\cal N}}

\newcommand{\PR}{{\cal P}}

\newcommand{\SR}{{\cal S}}

\newcommand{\R}{{\mathbb R}}
\newcommand{\RN}{{\mathbb R^N}}  
\newcommand{\bean}{\begin{eqnarray*}}
\newcommand{\eean}{\end{eqnarray*}}

\newcommand{\sbr}[1]{\left(#1\right)}
\newcommand{\mbr}[1]{\left[#1\right]}
\newcommand{\lbr}[1]{\left\{#1\right\}}
\newcommand{\abr}[1]{\left\langle#1\right\rangle}
\newcommand{\rd}{\mathrm d}

\newcommand{\nm}[1]{\Vert #1 \Vert}
\newcommand{\bu}{\bar U_\q}
\newcommand{\bv}{\bar V_\q}

\setitemize{itemindent=38pt,leftmargin=0pt,itemsep=-0.4ex,listparindent=26pt,partopsep=0pt,parsep=0.5ex,topsep=-0.25ex}

\numberwithin{equation}{section}

\begin{document}
\theoremstyle{plain}

\title{\bf Asymptotic behavior of positive solutions to the Lane-Emden system in dimension two
	\thanks{This work is partially supported by NSFC (No.12071240);  E-mails: zjchen2016@tsinghua.edu.cn, li-hw17@mails.tsinghua.edu.cn
			\& zou-wm@mail.tsinghua.edu.cn}
		}

\date{}
\author{
{\bf Zhijie Chen$^{1,2}$,\quad Houwang Li$^1$\quad \&\quad Wenming Zou$^1$}\\
\footnotesize \it 1. Department of Mathematical Sciences, Tsinghua University, Beijing 100084, China.\\
\footnotesize \it 2. Yau Mathematical Sciences Center, Tsinghua University, Beijing 100084, China.
		}

\maketitle
\begin{center}
\begin{minipage}{120mm}
\begin{center}{\bf Abstract }\end{center}
	Consider the Lane-Emden system
	$$\bcs
		-\Delta u=v^p,\quad u>0,\quad\text{in}~\Omega,\\
		-\Delta v=u^q,\quad v>0,\quad\text{in}~\Omega,\\
		u=v=0,\quad\text{on}~\pa\Omega,
	\ecs$$
	where $\Omega$ is a smooth bounded domain in $\mathbb{R}^N$ with $N\geq 2$ and $q\ge p>0$. The asymptotic behavior of {\it least energy solutions} of this system was studied in \cite{asy3-1,asy3-2} for $N\geq 3$. However, the case $N=2$ is different and remains completely open. In this paper, we study the case $N=2$ with
	$q=p+\q_p$ and $\sup_p\q_p<+\infty$. Under the following natural condition that holds for least energy solutions
		$$\limsup_{p\to+\infty} p\int_\Omega\nabla u_p\cdot\nabla v_p \rd x<+\infty,$$
	we give a complete description of the asymptotic behavior of {\it positive solutions} $(u_p,v_p)$ (i.e., not only for least energy solutions) as $p\to+\iy$. This seems the first result for asymptotic behaviors of the Lane-Emden system in the two dimension case.


\vskip0.23in

\end{minipage}
\end{center}

\vskip0.23in
\section{Introduction}

	We consider the Lane-Emden system
\be\lab{sys1}
	\bcs
		-\Delta u=v^p,\quad u>0,\quad\text{in}~\Omega,\\
		-\Delta v=u^q,\quad v>0,\quad\text{in}~\Omega,\\
		u=v=0,\quad\text{on}~\pa\Omega,
	\ecs
\ee
where $\Omega$ is a smooth bounded domain in $\mathbb{R}^N$ with $N\geq 2$ and $p,q>0$.
Without loss of generality, we always assume $q\geq p$.

The system \eqref{sys1} has received great interest and has been widely studied in the literature; see e.g. \cite{exist3,asy3-2,exist6,exist1,exist4,asy3-1,exist5,M1993}, the review article \cite{review1} and the references therein. It is well known that the existence and nonexistence of solutions of \eqref{sys1} is closely related to the so-called critical hyperbola
\[\f{1}{p+1}+\f{1}{q+1}=\f{N-2}{N}.\]
When
\[\f{1}{p+1}+\f{1}{q+1}>\f{N-2}{N},\]
which is known as the subcritical case, it was proved in \cite{exist6} by the degree method that \eqref{sys1} has a positive solution for $p,q>1$, and later it was proved in \cite{exist3} by the variational method that \eqref{sys1} has a least energy solution as long as $pq\neq 1$. However, if
\[\f{1}{p+1}+\f{1}{q+1}\leq\f{N-2}{N},\]
it was proved in \cite{M1993} that \eqref{sys1} has no solutions provided $\Omega$ is star-shaped.

In view of the above existence theory, a natural question that interests us arises:

\medskip

\noindent \textbf{Question}: {\it What is the limiting profile of positive solutions of \eqref{sys1} when $(p,q)$ converges to the critical hyperbola $\HR$? Here}
$$\HR:=\lbr{(p,q)\in\R^2:~p,q>0,~~\f{1}{p+1}+\f{1}{q+1}=\f{N-2}{N}}.$$
\medskip

It turns out that this problem is challenging and was only studied in \cite{asy3-1,asy3-2}. They considered the case $N\geq 3$ so that the critical hyperbola $\HR$ is a curve in $\R^2$. They studied the special case that
 $(p,q)=(p,q_\e)$ satisfies
\be\label{12}
	\f{1}{p+1}+\f{1}{q_\e+1}=\f{N-2}{N}+\e,\quad p\in \Big(\frac2{N-2},\frac{N+2}{N-2}\Big],\;q_\e,\e>0.
\ee
Then by letting $\e\to 0$ so that $(p, q_{\e})\to (p, q)\in \HR$, they proved that, among other things, the least energy solution $(u_\e,v_\e)$
of \eqref{sys1} must blow up, i.e., $\nm{u_\e}_\iy,\nm{v_\e}_\iy\to+\iy$, and a suitable scaling of $(u_\e,v_\e)$
converges to a positive solution of the Lane-Emden system in the whole space
$$
	\bcs
		-\Delta U=V^p,\quad U>0,\quad \text{in}~\RN,\\
		-\Delta V=U^q,\quad V>0,\quad \text{in}~\RN,\\
		U(x),V(x)\to0,\quad\text{as}~|x|\to+\iy.
	\ecs
$$
Furthermore, the blowup set of $(u_\e,v_\e)$ consists of a single point.
\vskip0.1in

As far as we know, the above problem for the case $N=2$ has not been studied in the literature and is completely open. Remark that in contrast with the case $N \geq 3$, $N=2$ implies that the critical hyperbola $\HR$ becomes the single point $(+\infty,+\infty)$, i.e., we have to consider $(p, q)\to (+\infty, +\infty)$, which makes the problem quite different comparing to the case $N\geq 3$.
The purpose of this paper is to give the first answer of this problem for the case $N=2$ under an additional assumption $|q-p|\leq C$ during $(p, q)\to (+\infty, +\infty)$. Thus, up to a subsequence, in this paper we always assume
that
\be\lab{para}
N=2,\quad	q=p+\q_p,\quad \q_p\ge0\quad \text{and}\quad \q_p\to\q \quad \text{as}~p\to+\iy.
\ee
Under the following natural condition
\be\lab{con1}
	\limsup_{p\to+\infty}p\int_\Omega\nabla u_p\cdot\nabla v_p\rd x\le \beta,
\ee
where $\beta$ is a positive constant,
we will give a complete description of the limiting profile of positive solutions $(u_p,v_p)$ as $p\to+\iy$.
The condition \eqref{con1} comes naturally from the bound state solution of \eqref{sys1},
i.e., solutions with finite energy.
In particular, we will prove in Section \ref{leastenergy} that the least energy solutions of \eqref{sys1} satisfy the condition \eqref{con1} with $\beta=8\pi e$.

Our another motivation comes from the study of asymptotic behaviors of positive solutions to the single Lane-Emden equation in dimension two.
In precise, let $p=q$ in \eqref{sys1}, we have
	$$\int_\Omega|\nabla (u_p-v_p)|^2=\int_\RN (v_p^p-u_p^p)(u_p-v_p)\le 0,$$
which means $u_p=v_p$. So system \eqref{sys1} with $N=2$ turns to be the single Lane-Emden equation
\be\lab{equ1}
	\bcs
	-\Delta u= u^p,\quad u>0,\quad\text{in}~\Omega\subset\R^2,\\
	u=0,\quad\text{on}~\pa\Omega,
	\ecs
\ee
and the condition \eqref{con1} turns to be
\be\lab{con2}
	\limsup_{p\to\infty}p\int_\Omega|\nabla u_p|^2\rd x\le \beta.
\ee
The asymptotic behavior of least energy solutions of \eqref{equ1} as $p\to\infty$ has been completely studied in \cite{asy1-1,asy1-2,asy1-3}, where it was proved that
least energy solutions $u_p$ of \eqref{equ1} satisfy the condition \eqref{con2} with $\beta=8\pi e$.
Recently, for high-energy positive solutions of \eqref{equ1}, i.e., under the condition \eqref{con2} with $\beta>8\pi e$,
the asymptotic behavior has been obtained in a series of papers \cite{asy2-1,asy2-2,asy2-3,asy2-4}.
Therefore, our study of the Lane-Emden system \eqref{sys1} in this paper can be seen as a generalization of these works to the system case.

To state our result, we need to introduce the following Liouville system which will appear as a limiting system of \eqref{sys1}
\be\lab{sys2}
	\bcs
		-\Delta U=e^V,\quad \text{in}~\R^2,\\
		-\Delta V=e^U,\quad \text{in}~\R^2,\\
		\int_{\R^2} e^U\rd x,\int_{\R^2} e^V\rd x<+\iy.
	\ecs
\ee
It was proved in \cite{classification} that for any solution of \eqref{sys2}, there must be
	$$U(x)=V(x)=\log\f{8\la^2}{(1+\la^2|x-y|^2)^2}$$
for some $\la>0$ and $y\in\R^2$. Denote $B_{r}(x)=\{y\in\R^2 \,|\, |y-x|<r\}$.

Here is our main result of this paper.

\bt\lab{thm1}
	Let $(u_p,v_p)$ be a family of positive solutions to \eqref{sys1} satisfying \eqref{para}-\eqref{con1}.
Then there exist $k\in\mathbb N\setminus\{0\}$ and a finite set of concentration points $\SR=\lbr{x_1,\cdots,x_k}\subset\Omega$ such that
	up to a subsequence, $(u_p,v_p)$ satisfies the following properties:
	for small $r>0$ and $i=1,\cdots,k$, set
		$$u_p(x_{i,p})=\max_{B_{2r}(x_i)}u_p\quad \text{and}\quad v_p(y_{i,p})=\max_{B_{2r}(x_i)}v_p,$$
	then
\begin{itemize}[fullwidth,itemindent=2em]
\item[(1)] 	$u_p(x_{i,p})\to\sqrt e$, $v_p(y_{i,p})\to\sqrt e$ as $p\to+\iy$. In particular, $\|u_p\|_{L^\infty(\Omega)}\to\sqrt e$ and $ \|v_p\|_{L^\infty(\Omega)}\to\sqrt e$.
\item[(2)]	$x_{i,p}\to x_i$, $y_{i,p}\to x_i$ and $\f{x_{i,p}-y_{i,p}}{\mu_{i,p}}\to0$  as $p\to+\iy$, where
				$$\mu_{i,p}^{-2}=pu_p^{p-1}(x_{i,p})\to+\iy.$$
\item[(3)]	as $p\to+\iy$,
 				$$pu_p(x),pv_p(x)\to 8\pi\sqrt e \sum_{i=1}^kG(x,x_i)\quad\text{in}~\CR_{loc}^2(\overline{\Omega}\setminus\SR),$$
 			where $G(x,y)$ is the Green function of $-\Delta$ in $\Omega$ under the Dirichlet boundary condition.
\item[(4)]	the concentration points $\{x_i\}_{i=1}^k$ satisfy
				$$\nabla R(x_i)-2\sum_{j=1,j\neq i}^k\nabla G(x_i,x_j)=0,\quad\text{for every $i=1,\cdots,k$},$$
			where $R(x)=H(x,x)$ is the Robin function and
				$H(x,y)=-\f{1}{2\pi}\log|x-y|-G(x,y)$
			is the regular part of $G(x,y)$.
\item[(5)]	let
			$$\bcs
				w_{i,p}(x)=\f{p}{u_p(x_{i,p})}(u_p(x_{i,p}+\mu_{i,p}x)-u_p(x_{i,p})),\\
				z_{i,p}(x)=\f{p}{u_p(x_{i,p})}(v_p(x_{i,p}+\mu_{i,p}x)-u_p(x_{i,p})),
			\ecs$$
			then $(w_{i,p},z_{i,p})\to (U_\q,V_\q)$ in $\CR_{loc}^2(\R^2)$ where $(U_\q+\f{\q}{2},V_\q)$ is a solution of the Liouville system \eqref{sys2} and
				$$U_\q(x)=\log\f{1}{(1+\f{1}{8}e^{\f{\q}{2}}|x|^2)^2},\quad
				V_\q(x)=\log\f{e^{\f{\q}{2}}}{(1+\f{1}{8}e^{\f{\q}{2}}|x|^2)^2},$$
			with $\q$ defined in \eqref{para}.
\item[(6)]  $p\int_\Omega\nabla u_p\cdot\nabla v_p\to k8\pi e$, $p\int_\Omega|\nabla u_p|^2\to k8\pi e$ and $p\int_\Omega|\nabla v_p|^2\to k8\pi e$ as $p\to+\iy$.
\end{itemize}
\et

\vskip 0.1in

\begin{remark} \
\begin{itemize}
\item[(1)] Our result is different from \cite{asy3-1,asy3-2} on two aspects. The first one is that comparing to the case $N\geq 3$, the limiting behavior for $N=2$ is quite different, and it turns out that $\nm{u_p}_\iy,\nm{v_p}_\iy\to\sqrt e$ and a suitable scaling of $(u_p,v_p)$ converges to a solution of the Liouville system \eqref{sys2}. The second one is that in contrast with \cite{asy3-1,asy3-2} where they only studied least energy solutions, here we also study high-energy positive solutions.

\item[(2)] Noting that in Theorem \ref{thm1}-(5), we scale $(u_p,v_p)$ at the local maximum points $x_{i,p}$ of $u_p$.
If we scale $(u_p,v_p)$ at the local maximum points $y_{i,p}$ of $v_p$, there is also a similar
convergence conclusion, see Remark \ref{converge4}. From condition \eqref{con1} and Theorem \ref{thm1}-(6), we see that
	$$k\le \mbr{\f{\beta}{8\pi e}}.$$
\end{itemize}
\end{remark}

Since we will prove in Section \ref{leastenergy} that the least energy solution $(u_p,v_p)$ of \eqref{sys1} satisfies
	$$\limsup_{p\to\infty}p\int_\Omega\nabla u_p\cdot\nabla v_p\le 8\pi e,$$
we immediately obtain

\bc\lab{thm2}
	The least energy solution $(u_p,v_p)$ of \eqref{sys1} must satisfy
		$$p\int_\Omega\nabla u_p\cdot\nabla v_p\to 8\pi e,\quad\text{as}~p\to\iy,$$
	and must behave as one-point concentration, i.e., $k=1$ in Theorem \ref{thm1}.
	Moreover, the concentration point $x_1$ is a critical point of the Robin function $R(x)$.
\ec

\vskip 0.1in
	In contrast with the single Lane-Emden equation \eqref{equ1}, since we deal with a system, we face additional difficulties.
The first one is how to construct the concentration set $\SR$.
For the single equation, to construct the concentration set, the most used idea is induction.
In precise, one can construct the concentration set following three steps (Let us take the single Lane-Emden equation $-\Delta u_p=u_p^p$ with $p\to\iy$ for example):
{\it
\vskip 0.1in
\begin{itemize}[fullwidth,itemindent=2em]
\item[ Step 1.]	One considers the global maximum point $x_{1,p}$ of $u_p$
	and proves that a suitable scaling of $u_p$ at $x_{1,p}$
	converges to a solution $U$ of the limiting equation $-\Delta U=e^U$ and $x_{1,p}\to x_1$, i.e. $x_1$ is a concentration point;

\vskip 0.05in
\item[ Step 2.]	Assume there exist $k$ families of points ${x_{i,p}}$, $1\leq i\leq k$, satisfying $\mu_{i,p}^{-2}=p u_{p}(x_{i,p})^{p-1}\to \infty$ and
    \begin{itemize}
    \item[$(\PR_k^1)$] $\lim_{p\to\infty}|x_{i,p}-x_{j,p}|/\mu_{j,p}=\infty$ for any $i\neq j$;
    \item[$(\PR_k^2)$] the scaling $p(u_p(x_{i,p}+\mu_{i,p}x)-u_p(x_{i,p}))/u_p(x_{i,p})$ converges to a solution $U$ of $-\Delta U=e^U$ and $x_{i,p}\to x_i$, i.e. $\SR_k=\lbr{x_1,\cdots,x_k}$ is a set of concentration points;
         \end{itemize}
       then one can prove that either the following $(\PR_k^3)$ holds or there exists another family of points $x_{k+1,p}$ satisfying $\mu_{k+1,p}^{-2}=p u_{p}(x_{k+1,p})^{p-1}\to \infty$, $x_{k+1,p}\to x_{k+1}$ such that $(\PR_{k+1}^1)$-$(\PR_{k+1}^2)$ hold
	for $\SR_{k+1}=\lbr{x_1,\cdots,x_{k+1}}$; where
\begin{itemize}
    \item[$(\PR_k^3)$]there exists $C>0$ such that
    \[\min_{1\leq i\leq k}|x-x_{i,p}|^2pu_p(x)^{p-1}\leq C,\quad\forall x\in \Omega.\]
    \end{itemize}

\vskip 0.05in
\item[ Step 3.]	By Steps 1-2, one can use the induction method. By the energy condition \eqref{con2},
	one shall prove that the induction must stop at finite steps, i.e., there exists a set $\SR_m=\{x_1,\cdots,x_m\}$ such that
	$(\PR_m^1)$-$(\PR_m^3)$ hold, but there is no other family of points $x_{m+1,p}$
	such that $(\PR_{m+1}^1)$-$(\PR_{m+1}^2)$ hold. Finally, one can prove that $\SR_m$ is the desired concentration set
	by $(\PR_m^1)$-$(\PR_m^3)$.
\end{itemize}
}
\vskip0.1in
\noindent
In early years, this idea was used to construct the bubble form decomposition of Palais-Smale sequence by Lions \cite{Lions-1,Lions-2}
and by Struwe \cite{Struwe-1} for high dimensions, which is the basis to study multi-bubble blow up behaviors of Lane-Emden equations
or Brezis-Nirenberg problems, see e.g. \cite{LiYY=1995,Kim=2016}.
After that, this idea was used to study multi-bubble blow up behaviors for many equations, such as
for two dimensional critical equations \cite{Druet-1,Druet-2},
for two dimensional Lane-Emden equations \cite{asy2-1,asy2-2},
for Lin-Ni's problem \cite{Druet-3}, etc.
However, for the Lane-Emden system \eqref{sys1}, it seems too difficult to find the proper induction conditions $(\PR_k^1)$, $(\PR_k^2)$ and $(\PR_k^3)$.
So we don't follow this idea to construct the concentration set $\SR$.
Instead, we will use PDE methods and the measure theory to prove that except for finite many points, $u_p,v_p$ converge to $0$ as $p\to\iy$.
Then we can conclude that the concentration set must be finite.
This is inspired by Wei and Ren's idea in \cite{asy1-1,asy1-2}, where they study the asymptotic behavior of least energy solution to Lane-Emden equation. It seems even new to construct the concentration set for single Lane-Emden equations.

\vskip0.1in
	Another difficulty is to compare two local maximums $\max_{B_{2r}(x_i)}u_p$ and $\max_{B_{2r}(x_i)}v_p$.
For $x_i\in\SR$, set
	$$u_p(x_{i,p})=\max_{B_{2r}(x_i)}u_p\quad \text{and}\quad v_p(y_{i,p})=\max_{B_{2r}(x_i)}v_p.$$
At first, we do not know whether to scale $(u_p,v_p)$ at $x_{i,p}$ or $y_{i,p}$, since we have no information about
the values of $\nm{u_p}_{L^\iy(B_{2r}(x_i))}$ and $\nm{v_p}_{L^\iy(B_{2r}(x_i))}$.
Moreover, noting that $q\ge p$, the roles of $u_p$ and $v_p$ in concentration are unequal,
which means that we can not directly assume $\nm{u_p}_{L^\iy(B_{2r}(x_i))}\ge \nm{v_p}_{L^\iy(B_{2r}(x_i))}$
or $\nm{u_p}_{L^\iy(B_{2r}(x_i))}\le \nm{v_p}_{L^\iy(B_{2r}(x_i))}$.
Fortunately, using the classification result of the Liouville system \eqref{sys2}, we can prove that
\be\lab{compare}
	\begin{aligned}
	0&\le \liminf_{p\to+\iy} p(\nm{v_p}_{L^\iy(B_{2r}(x_i))}-\nm{u_p}_{L^\iy(B_{2r}(x_i))})\\
	&\le \limsup_{p\to+\iy} p(\nm{v_p}_{L^\iy(B_{2r}(x_i))}-\nm{u_p}_{L^\iy(B_{2r}(x_i))})<+\iy,
	\end{aligned}
\ee
which promises the convergence in Theorem \ref{thm1}-(5); see Lemma \ref{converge3}.
We point out that our idea to handle the Lane-Emden system for $N=2$ is different from \cite{asy3-1,asy3-2} for $N\geq 3$.
In \cite{asy3-1,asy3-2}, since $p$ is fixed and $q_\e\to q$ in their assumption \eqref{12}, the Lane-Emden system is equivalent to a single equation
	$$\bcs 	
		-\Delta(-\Delta u_\e)^{\f{1}{p}}=u_\e^{q_\e},\quad u_\e>0,\quad\text{in}~\Omega,\\
		u_\e=\Delta u_\e=0,\quad\text{on}~\pa\Omega.		\ecs$$
Thus they can consider the scaling of $u_\e$, and then $v_\e$ is given by $v_\e=(-\Delta u_\e)^{\f{1}{p}}$.
Apparently, their idea is not suitable for our case $N=2$ due to $p,q\to+\iy$.

\vskip 0.1in
	This paper is organized as follows.
In Section 2, we give some preliminaries for the blow up analysis.
In Section 3, we prove the existence of a finite concentration set.
In Section 4, we use blow up techniques to prove Theorem \ref{thm1}.
Throughout the paper, we denote by $C, C_0, C_1, \cdots$ to be positive constants independent of $p$ but may be different in different places.

\vskip0.2in
\section{Preliminaries}

	First, we recall the Green function $G(x,y)$
of $-\Delta$ in $\Omega$ with the Dirichlet boundary condition:
\be
	\left\{		\begin{aligned}
	&-\Delta_x G(x,y)=\delta_y 	\quad &\text{in}~\Omega,\\
	&G(x,y)=0					\quad &\text{on}~\pa\Omega,
	\end{aligned} 	\right.
\ee
where $\delta_y$ is the Dirac function. It has the following form $$G(x,y)=-\f{1}{2\pi}\log|x-y|-H(x,y),\quad(x,y)\in\Omega\times\Omega,$$
where the function $H(x,y)$ is the regular part of $G(x,y)$.
It is well known that $H$ is a smooth function in $\Omega\times\Omega$, both
$G$ and $H$ are symmetric in $x$ and $y$, and there is some constant $C>0$ such that
\be|G(x,y)|\leq C|\log|x-y||,\qquad\forall \, x,y\in \Omega.\ee
 Let $R(x)=H(x,x)$
be the Robin function of $\Omega$.

For later usage, we need the following Pohozaev identities.

\bl\lab{Pohozaev}
	For any solution $(u,v)$ of \eqref{sys1}, $\Omega'\subset\Omega$ and $y\in\R^2$, we have
	\be\lab{pho1}
		\begin{aligned}
			&\f{2}{p+1}\int_{\Omega'}v^{p+1}\rd x+\f{2}{q+1}\int_{\Omega'}u^{q+1}\rd x\\
=&\int_{\pa\Omega'}\abr{\nabla u,\nu}\abr{\nabla v,x-y}+\abr{\nabla v,\nu}\abr{\nabla u,x-y}\rd\sigma_x\\
			&-\int_{\pa\Omega'}\abr{\nabla u,\nabla v}\abr{x-y,\nu}\rd \sigma_x	+\int_{\pa\Omega'}\abr{x-y,\nu}\sbr{\f{v^{p+1}}{p+1}+\f{u^{q+1}}{q+1}}\rd\sigma_x,
		\end{aligned}
	\ee
	\be\lab{pho2}
		\begin{aligned}
			-\int_{\pa\Omega'}\abr{\nabla u,\nu}\pa_iv+\abr{\nabla v,\nu}\pa_iu\rd\sigma_x
			+\int_{\pa\Omega'}\abr{\nabla u,\nabla v}\nu_i\rd \sigma_x
			=\int_{\pa\Omega'}\sbr{\f{v^{p+1}}{p+1}+\f{u^{q+1}}{q+1}}\nu_i\rd\sigma_x,
		\end{aligned}
	\ee	
	where $\nu=(\nu_1,\nu_2)$ is the outer normal vector of $\pa\Omega'$, and $i=1,2$.
\el
\bp
	By direct computations, we have
	$$\begin{aligned}
		&\quad -\Delta u(x-y)\cdot\nabla v-\Delta v(x-y)\cdot\nabla u\\
		&=-\operatorname{div}\sbr{\nabla u(x-y)\cdot\nabla v+\nabla v(x-y)\cdot\nabla u}
			+\nabla u\cdot\nabla \sbr{(x-y)\cdot\nabla v}+\nabla v\cdot\nabla \sbr{(x-y)\cdot\nabla u}\\
		&=-\operatorname{div}\sbr{\nabla u(x-y)\cdot\nabla v+\nabla v(x-y)\cdot\nabla u-(x-y)\nabla u\cdot\nabla v}.
	\end{aligned}$$
On the other hand,
	$$\begin{aligned}
		v^p(x-y)\cdot\nabla v+u^q(x-y)\cdot\nabla u =\operatorname{div}\sbr{(x-y)\Big(\f{v^{p+1}}{p+1}+\f{u^{q+1}}{q+1}\Big)}-2\sbr{\f{v^{p+1}}{p+1}+\f{u^{q+1}}{q+1}}.
	\end{aligned}$$
Then using \eqref{sys1} and the divergence theorem, we obtain \eqref{pho1}.

	Similarly, we have
	$$-\Delta u\pa_iv-\Delta v\pa_iu=-\operatorname{div}(\nabla u\pa_iv+\nabla v\pa_iu)+\pa_i(\nabla u\cdot\nabla v),$$	
and
	$$v^p\pa_iv+u^q\pa_iu=\pa_i\sbr{\f{v^{p+1}}{p+1}+\f{u^{q+1}}{q+1}}.$$
Then using \eqref{sys1} and the divergence theorem, we obtain \eqref{pho2}.
\ep

	Since $(u_p,v_p)$ satisfies \eqref{con1}, we have
\be
	p\int_\Omega u_p^{q+1}=p\int_\Omega v_p^{p+1}=p\int_\Omega\nabla u_p\cdot\nabla v_p\le \beta+o_p(1),
\ee
where $o_p(1)$ denotes those quantities satisfying $\lim_{p\to \infty}o_p(1)=0$.
Then by H\"older inequality,
\be\lab{est-v-1}
	\limsup_{p\to\iy} p\int_\Omega v_p^p
	\le \limsup_{p\to\iy} \sbr{p\int_\Omega v_p^{p+1}}^{\f{p}{p+1}}(p|\Omega|)^{\f{1}{p+1}}
	\le \beta,
\ee
\be\lab{est-u-1}
	\limsup_{p\to\iy} p\int_\Omega u_p^q
	\le \limsup_{p\to\iy} \sbr{p\int_\Omega u_p^{q+1}}^{\f{q}{q+1}}(p|\Omega|)^{\f{1}{q+1}}
	\le \beta.
\ee
Now we prove that $u_p,v_p$ are uniformly bounded as $p\to+\infty$.
For simplicity, we denote $\nm{u}_\iy=\nm{u}_{L^\iy(\Omega)}$.
\bl\lab{est-1}
	There exists a constant $C>0$ such that
	\be
		1\le\liminf_{p\to\iy}\nm{u_p}_\iy\le \limsup_{p\to\iy}\nm{u_p}_\iy\le C,
	\ee
	\be
		1\le\liminf_{p\to\iy}\nm{v_p}_\iy\le \limsup_{p\to\iy}\nm{v_p}_\iy\le C.
	\ee
\el
\bp
	Let $\vp_1>0$, $\nm{\vp_1}_\iy=1$, be the eigenfunction with respect
to the first eigenvalue $\la_1(\Omega)$ of $-\Delta$ in $\Omega$ with the Dirichlet boundary condition.
Then we have
$$
	\begin{aligned}
		&\quad \int_\Omega (v_p^{p-1}-\la_1(\Omega))v_p\vp_1+\int_\Omega (u_p^{q-1}-\la_1(\Omega))u_p\vp_1\\
		&=\int_\Omega (-\Delta u_p\vp_1+\Delta\vp_1 v_p)+(-\Delta v_p\vp_1+\Delta\vp_1 u_p)\\
		&=\int_\Omega (\nabla u_p\cdot\nabla \vp_1-\nabla v_p\cdot\nabla \vp_1)
					+(\nabla v_p\cdot\nabla \vp_1-\nabla u_p\cdot\nabla \vp_1)\\
		&=0.
	\end{aligned}
$$
So one of $\int_\Omega (v_p^{p-1}-\la_1(\Omega))v_p\vp_1$ and $\int_\Omega (u_p^{q-1}-\la_1(\Omega))u_p\vp_1$
is nonnegative. Without loss of generality, we assume $\int_\Omega (v_p^{p-1}-\la_1(\Omega))v_p\vp_1\ge0$,
which implies $\nm{v_p}_\iy\ge \la_1(\Omega)^{\f{1}{p-1}}$.
Meanwhile, by the Green representation formula,
\begin{align*}
v_p(x)=\int_{\Omega}G(x,y)u_p(y)^qdy\leq \nm{u_p}_\iy^q\int_{\Omega}|G(x,y)|dy\leq C\nm{u_p}_\iy^q,
\end{align*}
where $C$ is independent of $x\in\Omega$. Thus
	$$\nm{u_p}_\iy\ge (C^{-1}\nm{v_p}_\iy)^{\f{1}{q}}\ge C^{\f{-1}{q}}\la_1(\Omega)^{\f{1}{q(p-1)}}.$$
As a result, we obtain $\liminf_{p\to\iy}\nm{u_p}_\iy\ge1$ and $\liminf_{p\to\iy}\nm{v_p}_\iy\ge1$.

	To prove the second part, we recall \cite[Proposition 2.7]{asy1-4} that there is $C>0$ independent of $x\in\Omega$ such that $$\nm{G(x,\cdot)}_{L^{p+1}(\Omega)}^{p+1}\leq C(p+1)^{p+2},\quad \text{for $p$ large enough}.$$
Then
by the Green representation formula, for any $x\in\Omega$,
$$
	\begin{aligned}
		u_p(x)
		&=\int_\Omega G(x,y)v_p(y)^p\rd y\le \nm{G(x,\cdot)}_{L^{p+1}(\Omega)}\nm{v_p}_{L^{p+1}(\Omega)}^p\\
		&\le C(p+1)^{\f{p+2}{p+1}}\sbr{\f{C}{p}}^{\f{p}{p+1}}\le C,\quad\text{for $p$ large enough},
	\end{aligned}
$$
so $\limsup_{p\to\iy}\nm{u_p}_\iy\le C$.
Similarly, we can prove $\limsup_{p\to\iy}\nm{v_p}_\iy\le C$.
\ep

\section{Existence of the concentration set}
	
\subsection{The first concentration point}	
	We introduce some notations. Let $x_p,y_p\in\Omega$ be such that $u_p(x_p)=\nm{u_p}_\iy$ and $v_p(y_p)=\nm{v_p}_\iy$.
Take a subsequence if necessary, we assume
\be\lab{limit-1}
	x_p\to x_0,\quad y_p\to y_0,\quad\text{as}~p\to\iy,
\ee
and
\be\lab{limit-2}
	u_p(x_p)\to l,\quad v_p(y_p)\to\tilde l,\quad\text{as}~p\to\iy.
\ee
Then $x_0,y_0\in\bar\Omega$ and $l,\tilde l\ge1$. Let
\be
	\mu_p^{-2}=pu_p^{p-1}(x_p),
\ee
by Lemma \ref{est-1}, we know that $\mu_p\to0$.

	First, we prove that the maximum points $x_p,y_p$ are uniformly away from the boundary of $\Omega$.
\bl\lab{boundary}
	There exists $\delta_0>0$ independent of $p$ such that for $p$ large,
		$$dist(x_p,\pa\Omega),~ dist(y_p,\pa\Omega)\ge\delta_0.$$
\el
\bp
	We set
\be
	\bar u_p=\f{u_p}{\int_\Omega v_p^p},\quad \bar v_p=\f{v_p}{\int_\Omega u_p^q},
\ee
and
\be
	f_p=\f{v_p^p}{\int_\Omega v_p^p}, \quad g_p=\f{u_p^q}{\int_\Omega u_p^q}.
\ee
Then $\nm{f_p}_{L^1(\Omega)}=\nm{g_p}_{L^1(\Omega)}=1$ and
\be\lab{sys3}
	\bcs
		-\Delta\bar u_p=f_p,\quad \bar u_p>0,\quad\text{in}~\Omega,\\
		-\Delta\bar v_p=g_p,\quad \bar v_p>0,\quad\text{in}~\Omega,\\
		\bar u_p=\bar v_p=0,\quad\text{on}~\pa\Omega.
	\ecs
\ee
Since $\nm{f_p}_{L^1(\Omega)}=\nm{g_p}_{L^1(\Omega)}=1$, it follows from the elliptic $L^p$ estimate with the duality argument
(cf. \cite{p-estimate}) that $\bar u_p,\bar v_p$ are uniformly bounded in $W^{1,s}(\Omega)$
for any $1\le s<2$. In particular,
\begin{equation}\label{dd}\nm{\bar u_p}_{L^s(\Omega)}, \nm{\bar v_p}_{L^s(\Omega)}\leq C_s,\quad 1\leq s<2.\end{equation}

If $\Omega$ is a convex domain, by the moving plane method,
it was proved in \cite[Section 31.1, (31.19)]{book-1} that there exists a constant $\la_0>0$ such that
	$$\abr{\nabla\bar u_p(x),\nu(y)}\le 0, ~~\abr{\nabla\bar v_p(x),\nu(y)}\le 0, \quad\text{for}\quad y\in\pa\Omega,~x\in\Sigma(y,\la_0),$$
where $\Sigma(y,\la):=\lbr{x\in\Omega:~\abr{y-x,\nu(y)}\le\la}$ and $\nu(y)$ denotes the outer normal vector of $\partial \Omega$ at $y$.
Then by \cite[Lemma 13.2]{book-1}, there exist $\delta_0, C>0$ depending only on $\Omega$ and $\lambda_0$ such that
	$$\nm{\bar u_p}_{L^\iy(\omega)}\le C\int_\Omega\bar u_p\vp_1\rd x\le C\nm{\bar u_p}_{L^1(\Omega)}\le C,$$
$$\nm{\bar v_p}_{L^\iy(\omega)}\le C\int_\Omega\bar v_p\vp_1\rd x\le C\nm{\bar v_p}_{L^1(\Omega)}\le C,$$
where $\vp_1$ is given in Lemma \ref{est-1} and $$\omega=\lbr{x\in\Omega:~dist(x,\pa\Omega)\le \delta_0}.$$
While for $\Omega$ being not convex, as pointed out in \cite[Remark 31.5(ii)]{book-1}, by using the Kelvin transform we can also obtain
	\be\label{3-233}\nm{\bar u_p}_{L^\iy(\omega)},~\nm{\bar v_p}_{L^\iy(\omega)}\le C.\ee
As a result, we see from \eqref{est-v-1}-\eqref{est-u-1} that
\begin{equation}\label{boundary-1}
	\nm{u_p}_{L^\iy(\omega)}\le C\int_\Omega v_p^p\le \f{C}{p},
\quad
	\nm{v_p}_{L^\iy(\omega)}\le C\int_\Omega u_p^q\le \f{C}{p}.
\end{equation}
This, together with $\liminf_{p\to\iy}u_p(x_p),\liminf_{p\to\iy}v_p(y_p)\ge 1$, implies $x_p,y_p\not\in\omega$ for $p$ large.
\ep


Define the scaling of $u_p,v_p$ at $x_p$ by
\be
	\begin{aligned}
	w_p(x)&=\f{p}{u_p(x_p)}\sbr{u_p(x_p+\mu_px)-u_p(x_p)},\\
	z_p(x)&=\f{p}{u_p(x_p)}\sbr{v_p(x_p+\mu_px)-u_p(x_p)},
	\end{aligned}\quad\quad x\in\Omega_p:=\f{\Omega-x_p}{\mu_p}.
\ee
Lemma \ref{boundary} implies $x_0,y_0\in\Omega$ and so $\Omega_p\to\R^2$ as $p\to\iy$.
It is easy to check that
	$$1+\f{w_p(x)}{p}=\f{u_p(x_p+\mu_px)}{u_p(x_p)},\quad 1+\f{z_p(x)}{p}=\f{v_p(x_p+\mu_px)}{u_p(x_p)},$$
and (note $\mu_p^{-2}=pu_p^{p-1}(x_p)$ and $q=p+\theta_p$)
\be
	\bcs
		-\Delta w_p=(1+\f{z_p}{p})^p, \quad\text{in}~\Omega_p,\\
		-\Delta z_p=u_p(x_p)^{\q_p}(1+\f{w_p}{p})^q, \quad\text{in}~\Omega_p,\\
		w_p=z_p=-p,\quad\text{on}~\pa\Omega_p.
	\ecs
\ee
Moreover, $$w_p(0)=0=\max_{\Omega_p}w_p.$$

As pointed out in Section 1, one difficulty is the lack of the relationship between the values $u_p(x_p)=\|u_p\|_{\infty}$ and $v_p(y_p)=\|v_p\|_{\infty}$ at the moment.
To settle it, we give the following important lemma.

\bl\lab{converge1}
	Suppose $\limsup_{p\to\iy}p(v_p(y_p)-u_p(x_p))\le C$ for some $C\ge 0$. Then up to a subsequence of $p\to\infty$,
		$$w_p\to \bu,\quad z_p\to \bv,\quad\text{in}~\CR_{loc}^2(\R^2),$$
	where $l=\lim_{p\to\infty}u_p(x_p)$, $(\bu+\q\log l,\bv)$ is a solution of \eqref{sys2} and
	\be
		\bar U_\q(x)=\log\f{1}{(1+\f{1}{8}l^{\q}|x|^2)^2},\quad
		\bar V_\q(x)=\log\f{l^{\q}}{(1+\f{1}{8}l^{\q}|x|^2)^2}.
	\ee
\el
\bp
	For any $R>0$, we have $B_R(0)\subset\Omega_p$ for $p$ large.
Take $w_p=w_{p,1}+w_{p,2}$ in $B_R(0)$ with
$$
	\bcs
		-\Delta w_{p,1}=(1+\f{z_p}{p})^p,\quad\text{in}~B_R(0),\\
		w_{p,1}=0,\quad\text{on}~\pa B_R(0),
	\ecs
$$
and
$$
	\bcs
		-\Delta w_{p,2}=0,\quad\text{in}~B_R(0),\\
		w_{p,1}=w_p,\quad\text{on}~\pa B_R(0).
	\ecs
$$
Since $\limsup_{p\to\iy}p(v_p(y_p)-u_p(x_p))\le C$ and $\liminf_{p\to\iy}u_p(x_p)\geq 1$ imply $z_p(x)\leq C$, we obtain
	$$0\le \Big(1+\f{z_p}{p}\Big)^p\le \Big(1+\f{C}{p}\Big)^p\le C,$$
which means $0\le -\Delta w_{p,1}\le C$.
Thus $\nm{w_{p,1}}_{L^\iy(B_R(0))}\le C$ and $w_{p,2}=w_p-w_{p,1}\le C$.
By Harnack inequality, we see that if $\inf_{B_R(0)}w_{p,2}\to-\iy$, then $\sup_{B_R(0)}w_{p,2}\to-\iy$ as $p\to\iy$,
which contradicts to $w_{p,2}(0)=w_p(0)-w_{p,1}(0)\ge -C$.
So we have $\nm{w_{p,2}}_{L^\iy(B_R(0))}\le C$, and hence $\nm{w_p}_{L^\iy(B_R(0))}\le C$.
Then by elliptic estimates (see \cite[Theorems 3.9 and 8.8]{book-2}), we have
	$$\nm{\nabla w_p}_{L^\iy(B_R(0))}\le C,\quad\text{and}\quad \nm{w_p}_{W^{2,2}(B_R(0))}\le C.$$
Up to a subsequence, we can assume $w_p\to w$ in $\CR(B_R(0))$ and $w_p\rh w$ in $W^{2,2}(B_R(0))$.

Similarly, we take $z_p=z_{p,1}+z_{p,2}$ with $z_{p,1}=0$ on the boundary $\pa B_R(0)$ and $z_{p,2}$ is harmonic in $B_R(0)$.
Since $u_p(x_p)\to l$ and $\q_p\to\q$, we have $0\le -\Delta z_{p,1}=u_p(x_p)^{\q_p}(1+\f{w_p}{p})^q\le C$, which yields $\nm{z_{p,1}}_{L^\iy(B_R(0))}\le C$.
From here and $z_p\le C$, we obtain $z_{p,2}=z_p-z_{p,1}\le C$. We will prove by contradiction that $z_{p,2}\ge -C$.
Suppose $\inf_{B_R(0)}z_{p,2}\to-\iy$ as $p\to\infty$, then $\sup_{B_R(0)}z_{p,2}\to-\iy$ by Harnack inequality.
This implies $\sup_{B_R(0)}z_p\le \sup_{B_R(0)}z_{p,1}+\sup_{B_R(0)}z_{p,2}\to-\iy$,
so $\sup_{B_R(0)}(1+\f{z_p}{p})^p\to0$, i.e., $\nm{\Delta w_p}_{L^\iy(B_R(0))}\to0$ as $p\to\iy$.
Thus in the distribution sense, $-\Delta w=0$ in $B_R(0)$.
By choosing $R=R_n\to\iy$, repeating the above process, we can obtain that up to a subsequence,
$w_p\to w$ in $\CR_{loc}(\R^2)$ and $w_p\rh w$ in $W^{2,2}_{loc}(\R^2)$, where
	$$-\Delta w=0\quad\text{in}~\R^2,\quad w\le w(0)=0.$$
By Liouville's Theorem, $w\equiv 0$. However, by Fatou's Lemma,
\begin{align}\lab{tem-1}
	\int_{\R^2}e^w
	&\le \liminf_{p\to\iy}\int_{\Omega_p}\Big(1+\f{w_p}{p}\Big)^p\rd x
	=\liminf_{p\to\iy}\f{p}{u_p(x_p)}\int_{\Omega}u_p^p\rd x\nonumber\\
&\le \limsup_{p\to\iy}\Big(p\int_{\Omega}u_p^{q+1}\Big)^{\f{p}{q+1}}(p|\Omega|)^{\f{\theta_p+1}{q+1}}\le C,
\end{align}
which contradicts to $w=0$.
So we prove that $\inf_{B_R(0)}z_{p,2}\ge -C$ and then $\nm{z_p}_{L^\iy(B_R(0))}\le C$.
Then by elliptic estimates, we have $\nm{\nabla z_p}_{L^\iy(B_R(0))}\le C$ and $\nm{z_p}_{W^{2,2}(B_R(0))}\le C$.

Since $R>0$ is arbitrary, the above argument shows that $w_p, z_p$ are uniformly bounded in $W_{loc}^{2,2}(\R^2)\cap \CR_{loc}^{1}(\R^2)$. Then by the Schauder estimates, we conclude that $w_p, z_p$ are uniformly bounded in $\CR_{loc}^{2,\alpha}(\R^2)$, and so
 up to a subsequence, we may assume
$w_p\to w$ and $z_p\to z$ in $\CR_{loc}^2(\R^2)$ with
$$
	\bcs
	-\Delta w=e^z,\quad\text{in}~\R^2,\\
	-\Delta z=l^\q e^w,\quad\text{in}~\R^2.
	\ecs
$$
Similar to \eqref{tem-1}, we have $\int_{\R^2}e^w<\iy$ and $\int_{\R^2}e^z<\iy$.
Thus $(w+\q\log l,z)$ is a solution of \eqref{sys2}, and hence
	$$w(x)+\q\log l=z(x)=\log\f{8\la^2}{(1+\la^2|x-y|^2)^2}$$
for some $\la>0$ and $y\in\R^2$.
Finally, by $w(0)=0=\max_{\R^2} w$, we obtain $y=0$ and $\lambda^2=\frac18 l^{\theta}$, so $w=\bu$ and $z=\bv$. The proof is complete.
\ep

\br\lab{converge2}
	Let $\tilde \mu_p^{-2}=pv_p^{p-1}(y_p)\to\iy$ and
define the scaling of $u_p,v_p$ at $y_p$ by
\be
	\begin{aligned}
	\tilde w_p(x)&=\f{p}{v_p(y_p)}\sbr{u_p(y_p+\tilde \mu_px)-v_p(y_p)},\\
	\tilde z_p(x)&=\f{p}{v_p(y_p)}\sbr{v_p(y_p+\tilde \mu_px)-v_p(y_p)},
	\end{aligned}\quad\quad x\in\tilde\Omega_p:=\f{\Omega-y_p}{\tilde\mu_p}.
\ee
By Lemma \ref{boundary}, we see $\tilde\Omega_p\to\R^2$ as $p\to\iy$.
It is easy to check that
	$$1+\f{\tilde w_p(x)}{p}=\f{u_p(y_p+\tilde\mu_px)}{v_p(y_p)},
		\quad 1+\f{\tilde z_p(x)}{p}=\f{v_p(y_p+\tilde\mu_px)}{v_p(y_p)},$$
and
\be
	\bcs
		-\Delta\tilde w_p=(1+\f{\tilde z_p}{p})^p, \quad\text{in}~\tilde\Omega_p,\\
		-\Delta\tilde z_p=v_p(y_p)^{\q_p}(1+\f{\tilde w_p}{p})^q, \quad\text{in}~\tilde\Omega_p,\\
		\tilde w_p=\tilde z_p=-p,\quad\text{on}~\pa\tilde\Omega_p.
	\ecs
\ee
Moreover, $\tilde z_p(0)=0=\max_{\tilde \Omega_p}\tilde z_p$. Suppose $\limsup_{p\to\iy}p(u_p(x_p)-v_p(y_p))\le C$ for some $C\ge 0$. Then $\tilde w_p\leq C$ in $\tilde\Omega_p$. Consequently,
we can follow the proof of Lemma \ref{converge1} and obtain that up to a subsequence,
		$$\tilde w_p\to \tilde U_\q,\quad \tilde z_p\to \tilde V_\q,\quad\text{in}~\CR_{loc}^2(\R^2),\quad\text{as}~p\to\iy,$$
	where $\tilde{l}:=\lim_{p\to\infty}v_p(y_p)$, $(\tilde U_\q+\q\log \tilde l,\tilde V_\q)$ is a solution of \eqref{sys2} and
	\be
		\tilde U_\q(x)=\log\f{\tilde l^{-\q}}{(1+\f{1}{8}|x|^2)^2},\quad
		\tilde V_\q(x)=\log\f{1}{(1+\f{1}{8}|x|^2)^2}.
	\ee
\er

\vskip0.2in

	Now we compare the maximum values of $u_p,v_p$.
\bl\lab{compare1}
	There exists a constant $C\ge 0$ such that
		$$0\le \liminf_{p\to+\iy} p(v_p(y_p)-u_p(x_p))\le \limsup_{p\to+\iy} p(v_p(y_p)-u_p(x_p))\le C.$$
\el
\bp
	Assume by contradiction that $p(v_p(y_p)-u_p(x_p))\to d<0$ up to a subsequence, where $d$ might be $-\infty$.
Applying Lemma \ref{converge1}, we see that $w_p\to\bu$ and $z_p\to\bv$ in $\CR^2_{loc}(\R^2)$ as $p\to\iy$.
Then
	$$\max z_p\ge z_p(0)\to \bv(0)=\q\log l\ge 0.$$
However,
	$$\max z_p=\f{p(v_p(y_p)-u_p(x_p))}{u_p(x_p)}\to \f{d}{l}<0,$$
which is a contradiction. This proves $\liminf_{p\to+\iy} p(v_p(y_p)-u_p(x_p))\ge 0$.

Similarly, to prove $\limsup_{p\to+\iy} p(v_p(y_p)-u_p(x_p))\leq C$, we assume by contradiction that $p(v_p(y_p)-u_p(x_p))\to+\infty$ up to a subsequence. Then
we can apply
Remark \ref{converge2}, which asserts that
$\tilde w_p\to\tilde U_\q$ and $\tilde z_p\to\tilde V_\q$ in $\CR^2_{loc}(\R^2)$ up to a subsequence.
Consequently,
	$$-\infty\leftarrow\f{p(u_p(x_p)-v_p(y_p))}{v_p(y_p)}=\max\tilde w_p\ge \tilde w_p(0)\to \tilde U_\q(0)=-\q\log\tilde l,$$
again a contradiction.
\ep

\br
	Lemma \ref{compare1} shows that the assumptions in Lemma \ref{converge1} and Remark \ref{converge2} hold automatically, so the conclusions of both Lemma \ref{converge1} and Remark \ref{converge2} always hold, and $l=\tilde l$.
\er

\bl\lab{est-6}
	It holds	
		$$\f{C_1}{p}\le \int_\Omega u_p^q,~\int_\Omega v_p^p\le \f{C_2}{p},$$
	for $p$ large enough, where $C_1,C_2>0$ .
\el
\bp
	From \eqref{est-v-1}-\eqref{est-u-1}, we have
$\int_\Omega u_p^q,~\int_\Omega v_p^p\le \f{C_2}{p}$ for $p$ large.
On the other hand, it follows from Lemma \ref{converge1} that
	$$\liminf_{p\to\iy}p\int_\Omega u_p^q
		=\liminf_{p\to\iy}u_p(x_p)^{\q_p+1}\int_{\Omega_p}\Big(1+\f{w_p}{p}\Big)^q
		\ge l^{\q+1}\int_{\R^2}e^{\bu}=8\pi l,$$
	$$\liminf_{p\to\iy}p\int_\Omega v_p^p
		=\liminf_{p\to\iy}u_p(x_p)\int_{\Omega_p}\Big(1+\f{z_p}{p}\Big)^p
		\ge l\int_{\R^2}e^{\bv}=8\pi l,$$
which means $\int_\Omega u_p^q,~\int_\Omega v_p^p\ge \f{C_1}{p}$ for $p$ large.
\ep

\begin{corollary}\lab{est-6-0}
Define
$$L_p:=\max\left\{\frac{p\int_{\Omega}u_p^q}{e(\int_{\Omega}v_p^p)^{1/p}},
\frac{q\int_{\Omega}v_p^p}{e(\int_{\Omega}u_p^q)^{1/q}}\right\}\quad\text{and}\quad L:=\limsup_{p\to\infty}L_p.$$
Then $\frac{8\pi l}{e}\leq L\leq \frac{\beta}{e}$.
\end{corollary}

\begin{proof}
This is a direct consequence of the proof of Lemma \ref{est-6}.
\end{proof}

\subsection{Finiteness of the concentration points}
	Let $\bar u_p,\bar v_p$ and $f_p,g_p$ be defined as in the proof Lemma \ref{boundary}. 
Since $\nm{f_p}_{L^1(\Omega)}=\nm{g_p}_{L^1(\Omega)}=1$, we assume (up to a subsequence) that
	$$f_p\rh \mu,\quad g_p\rh\nu,\quad\text{weakly in}~\MR(\Omega),\quad\text{as}~p\to\iy,$$
where $\MR(\Omega)$ is the space of Radon measures. Obviously, $\mu(\Omega)=\nu(\Omega)=1$.
	
	For any $\delta>0$, we say a point $x_*\in\Omega$ is a $\delta$-regular point with respect to $\mu$ (resp. $\nu$),
if there exists a $\vp\in\CR_0(\Omega)$ such that $0\le\vp\le 1$, $\vp\equiv1$ near $x_*$ and
	$$\int_\Omega\vp\rd\mu<\f{4\pi}{L+2\delta} \quad
		\sbr{\text{resp.}\int_\Omega\vp\rd\nu<\f{4\pi}{L+2\delta}},$$
where $L$ is defined in Corollary \ref{est-6-0}.
Denote
	$$\Sigma_\mu(\delta):=\lbr{x\in\Omega:~x~\text{is not a $\delta$-regular point w.r.t. $\mu$}~},$$
	$$\Sigma_\nu(\delta):=\lbr{x\in\Omega:~x~\text{is not a $\delta$-regular point w.r.t. $\nu$}~}.$$
Before proceeding our discussion, we quote a $L^1$ estimates from \cite{1-estimate}.
\bl[\cite{1-estimate}]\lab{tem-2}
	Let $u$ be a solution of
	$$\bcs
		-\Delta u=f\quad\text{in}~\Omega,\\
		u=0\quad\text{on}~\pa\Omega,
	\ecs$$
	where $\Omega$ is a smooth bounded domain in $\R^2$. Then for any $0<\e<4\pi$, we have
		$$\int_{\Omega}\exp\sbr{\f{(4\pi-\e)|u(x)|}{\nm{f}_{L^1(\Omega)}}}\rd x\le \f{4\pi^2}{\e}(\mathrm{diam}~\Omega)^2.$$
\el

We give an equivalent characterization of the sets $\Sigma_\mu(\delta)$ and $\Sigma_\nu(\delta)$, whose proof is inspired by Wei and Ren's idea in \cite{asy1-1,asy1-2}.
\bl\lab{equiv-1}
	For any $\delta>0$, we have
	\begin{itemize}[fullwidth,itemindent=2em]
	\item[(1)]	$x_*\in\Sigma_\mu(\delta)$ if and only if for any $R>0$,
					it holds $\nm{\bar v_p}_{L^\iy(B_R(x_*))}\to+\iy$ as $p\to\iy$;
	\item[(2)]	$x_*\in\Sigma_\nu(\delta)$ if and only if for any $R>0$,
					it holds $\nm{\bar u_p}_{L^\iy(B_R(x_*))}\to+\iy$ as $p\to\iy$.
	\end{itemize}
Consequently, neither $\Sigma_\mu(\delta)$ nor $\Sigma_\nu(\delta)$ depend on the choice of $\delta$.
\el
\bp
	We only prove (1), since (2) can be proved similarly.
First, take $x_*\not\in\Sigma_\mu(\delta)$, we want to prove that there exists $R_0>0$ such that
$\nm{\bar v_p}_{L^\iy(B_{R_0}(x_*))}\le C$ as $p\to\iy$.
To this goal, we recall \eqref{dd} in the proof of Lemma \ref{boundary} that $\nm{\bar v_p}_{L^{3/2}(\Omega)}\le C$.
We claim that  there exist small $R_0>0$ and $\delta_0>0$ such that
\be\lab{claim-1}
	\nm{g_p}_{L^{1+\delta_0}(B_{2R_0}(x_*))}\le C,\quad\text{as}~p\to\iy.
\ee
Once \eqref{claim-1} is proved, we can apply the weak Harnack inequality (\cite[Theorem 8.17]{book-2}) to obtain
	$$\nm{\bar v_p}_{L^\iy(B_{R_0}(x_*))} \le C\sbr{\nm{\bar v_p}_{L^{3/2}(B_{2R_0}(x_*))}
			+ \nm{g_p}_{L^{1+\delta_0}(B_{2R_0}(x_*))}}\le C. $$

Now we need to check the claim \eqref{claim-1}.
Let
	$$\al_p(x)=\f{u_p(x)}{(\int_\Omega u_p^q)^{1/q}}=\frac{\int_\Omega v_p^p}{(\int_\Omega u_p^q)^{1/q}}\bar{u}_p(x).$$
Since $\f{\log x}{x}\le\f{1}{e}$ for any $x\in(0,+\iy)$, we obtain
	$$\log \al_p(x)\le \f{1}{e}\al_p(x),\quad \forall x. $$
Therefore, for any $x\in\Omega$
$$
	\begin{aligned}
		g_p(x)
		&=\f{u_p(x)^q}{\int_\Omega u_p^q}=\exp\sbr{q\log\al_p(x)}\le \exp\sbr{\f{q}{e}\al_p(x)}\\
		&\le \exp\sbr{\frac{q\int_\Omega v_p^p}{e(\int_\Omega u_p^q)^{1/q}}\bar{u}_p(x)}
			\le \exp\sbr{L_p \bar u_p(x)},
	\end{aligned}
$$
where $L_p$ is defined in Corollary \ref{est-6-0}. Since $L=\limsup_{p\to\infty}L_p$ and  $\delta>0$, we have
	$$g_p(x)\le \exp\sbr{(L+\tfrac{\delta}{2}) \bar u_p(x)}\quad\text{for $p$ large}.$$
Since $x_*\not\in\Sigma_\mu(\delta)$, it follows from the definition of $\delta$-regular points that there exists $R_1>0$ such that $B_{2R_1}(x_*)\subset \Omega$ and
	$$\int_{B_{2R_1}(x_*)}f_p<\f{4\pi}{L+\delta}\quad\text{for $p$ large}.$$
Take $\bar u_p=\bar u_{p,1}+\bar u_{p,2}$ with $\bar u_{p,1}=0$ on the boundary $\pa B_{2R_1}(x_*)$
and $\bar u_{p,2}$ is harmonic in $B_{2R_1}(x_*)$.
By the maximum principle, $\bar u_{p,1}>0$ and $\bar u_{p,2}>0$ in $B_{2R_1}(x_*)$.
By Lemma \ref{tem-2}, we have
	$$\int_{B_{2R_1}(x_*)}\exp\sbr{\f{\ga \bar{u}_{p,1}(x)}{\nm{f_p}_{L^1(B_{2R_1}(x_*))}}}\rd x
			\le C_\ga,\quad\text{for any $\ga\in(0,4\pi)$}.$$
Note that $0<\bar u_{p,2}< \bar u_p$ in $B_{2R_1}(x_*)$. Then by the mean value theorem for harmonic functions and \eqref{dd}, we obtain
	$$\nm{\bar u_{p,2}}_{L^\iy(B_{R_1}(x_*))}\le C\nm{\bar u_{p,2}}_{L^1(B_{2R_1}(x_*))}\le C\nm{\bar u_{p}}_{L^1(B_{2R_1}(x_*))}
			\le C\nm{\bar u_{p}}_{L^1(\Omega)}\le C.$$
Take $\delta_0>0$ such that
$\gamma:=4\pi(1+\delta_0)\frac{L+\frac{\delta}{2}}{L+\delta}<4\pi$. Then by the above estimates, we conclude that for $p$ large,
$$
	\begin{aligned}
		\int_{B_{R_1}(x_*)} g_p(x)^{1+\delta_0}\rd x
		&\le \int_{B_{R_1}(x_*)}\exp\sbr{(1+\delta_0)(L+\tfrac{\delta}{2}) \bar u_p(x)} \rd x\\
		&\le C\int_{B_{R_1}(x_*)}\exp\sbr{(1+\delta_0)(L+\tfrac{\delta}{2}) \bar u_{p,1}(x)} \rd x\\
&\le C\int_{B_{2R_1}(x_*)}\exp\sbr{(1+\delta_0)(L+\tfrac{\delta}{2}) \bar u_{p,1}(x)} \rd x\\
		&\le C\int_{B_{2R_1}(x_*)}\exp\sbr{4\pi(1+\delta_0)\f{L+\f{\delta}{2}}{L+\delta}
				\f{\bar{u}_{p,1}(x)}{\nm{f_p}_{L^1(B_{2R_1}(x_*))}}}\rd x\\
		&=C\int_{B_{2R_1}(x_*)}\exp\sbr{
				\f{\gamma\bar{u}_{p,1}(x)}{\nm{f_p}_{L^1(B_{2R_1}(x_*))}}}\rd x\le C_{\gamma}.
	\end{aligned}
$$
Thus by choosing $R_0=R_1/2$, we finish the proof of the claim \eqref{claim-1}.

	Finally, given any $x_*\in\Sigma_\mu(\delta)$, we claim that
for any $R>0$, $\nm{\bar v_p}_{L^\iy(B_R(x_*))}\to+\iy$ as $p\to\iy$.
If not, there exists $R_1>0$ such that up to a subsequence, $\nm{\bar v_p}_{L^\iy(B_{R_1}(x_*))}\le C$ as $p\to\iy$.
Then by Lemma \ref{est-6},
	$$v_p(x)=\bar v_p(x) \int_\Omega u_p^q\le C \int_\Omega u_p^q\le \f{C}{p}\quad\text{for}~x\in B_{R_1}(x_*),$$
and hence
	$$\int_{B_{R_1}(x_*)}f_p=\frac{\int_{B_{R_1}(x_*)}v_p^p}{\int_{\Omega}v_p^p}\le Cp\int_{B_{R_1}(x_*)}\sbr{\f{C}{p}}^p \to 0\quad\text{as $p\to\iy$.}$$
Thus by the definition, $x_*\not\in\Sigma_\mu(\delta)$, a contradiction. This finishes the proof.
\ep

	Actually we have

\bl\lab{equiv-2}
	For any $\delta>0$, $\Sigma_\mu(\delta)=\Sigma_\nu(\delta)$ is a finite nonempty set.
\el
\bp
	Since $x_p\to x_0$ and $\liminf_{p\to\iy}u_p(x_p)\ge 1$, we see from Lemma \ref{est-6} that \be\label{3-19}\bar u_p(x_p)=\frac{u_p(x_p)}{\int_{\Omega}v_p^p}\geq Cpu_p(x_p)\to\iy\quad\text{as }p\to\infty,\ee
so $x_0\in \Sigma_\nu(\delta)$ by Lemma \ref{equiv-1}.
Similarly, we have $y_0\in\Sigma_\mu(\delta)$.
Furthermore,
	$$1=\mu(\Omega)\ge \f{4\pi}{L+2\delta}\#\Sigma_\mu(\delta),\quad
	1=\nu(\Omega)\ge \f{4\pi}{L+2\delta}\#\Sigma_\nu(\delta),$$
which implies $1\le \#\Sigma_\mu(\delta),\#\Sigma_\nu(\delta)<\iy$.

	It remains to prove $\Sigma_\mu(\delta)=\Sigma_\nu(\delta)$.
For a point $x_*\in\Sigma_\mu(\delta)$, take $r_0>0$ small such that $B_{2r_0}(x_*)\cap\Sigma_\mu(\delta)=\{x_*\}$.
Fix any $0<r\leq r_0$. Since $-\Delta \bar v_p=g_p$, it follows from \cite[Theorem 3.7]{book-2} that
\be\label{3-18}\max_{B_r(x_*)} \bar v_p\le C_r(\max_{\pa B_r(x_*)}\bar v_p + \max_{B_r(x_*)} g_p ).\ee
Note from Lemma \ref{equiv-1} that
	$$\max_{B_r(x_*)} \bar v_p\to\iy,\quad\text{as}~p\to\iy.$$
 We claim that
	\begin{equation}\label{3-17}\max_{\pa B_r(x_*)} \bar v_p\le C_r,\quad\text{as}~p\to\iy.\end{equation}
Indeed, if \eqref{3-17} does not hold, then up to a subsequence, there exists $\tilde{x}_p\in\pa B_r(x_*)$ such that $\bar v_p(\tilde{x}_p)\to\iy$ and $\tilde{x}_p\to \tilde x_*\in \pa B_r(x_*)$ as $p\to\iy$. Then it follows from Lemma \ref{equiv-1} that $\tilde x_*\in\Sigma_\mu(\delta)\setminus\{x_*\}$,
which contradicts with $B_{2r_0}(x_*)\cap\Sigma_\mu(\delta)=\{x_*\}$. Thus \eqref{3-17} holds. Clearly \eqref{3-18}-\eqref{3-17} imply \be\label{3-20} \max_{B_r(x_*)} g_p\to\iy\quad\text{ as $p\to\iy$},\ee which, together with $g_p(x)=\frac{u_p(x)^q}{\int_{\Omega}u_p^q}\leq Cp u_p(x)^q$, yields
	$$\liminf_{p\to\iy}\max_{B_r(x_*)} u_p\ge 1.$$
Thus the same argument as \eqref{3-19} implies $\max_{B_r(x_*)} \bar u_p\to\iy$ as $p\to\iy$ for any $0<r\le r_0$,
so it follows from Lemma \ref{equiv-1} that $x_*\in\Sigma_\nu(\delta)$.
This proves $\Sigma_\mu(\delta)\subset\Sigma_\nu(\delta)$.
Similarly we can prove $\Sigma_\nu(\delta)\subset\Sigma_\mu(\delta)$, so $\Sigma_\mu(\delta)=\Sigma_\nu(\delta)$.
\ep

\bl\lab{tem-11}
	For any compact subset $K\subset\overline{\Omega}\setminus\Sigma_\mu(\delta)$, we have
		$$\nm{u_p}_{L^\iy(K)},~\nm{v_p}_{L^\iy(K)}\le \f{C}{p},\quad\text{for $p$ large}. $$
\el
\bp
	Let $\omega=\lbr{x\in\Omega:~dist(x,\pa\Omega)\le \delta_0}$ be the one given in Lemma \ref{boundary}.
Then by \eqref{boundary-1},
	$$\nm{u_p}_{L^\iy(w)},~\nm{v_p}_{L^\iy(w)}\le \f{C}{p},\quad\text{for $p$ large}. $$
Thus it suffices to prove this lemma for those compact subsets $K\Subset\Omega\setminus\Sigma_\mu(\delta)$. Given any $K\Subset\Omega\setminus\Sigma_\mu(\delta)$ and
for any $x\in K$, we have $x\notin \Sigma_\mu(\delta)=\Sigma_\nu(\delta)$, then it follows from the proof of
 Lemma \ref{equiv-1} that there exists $R_x>0$ such that $B_{R_x}(x)\subset\Omega$ and
\[\nm{\bar u_p}_{L^\iy(B_{R_x}(x))},~\nm{\bar v_p}_{L^\iy(B_{R_x}(x))}\le C,\quad\text{for $p$ large}.\]
From here and the finite covering theorem, we obtain
$$\nm{\bar u_p}_{L^\iy(K)},~\nm{\bar v_p}_{L^\iy(K)}\le C,\quad\text{for $p$ large}. $$
This, together with Lemma \ref{est-6} which says that \[u_p(x)=\bar u_p(x)\int_\Omega v_p^p\leq \frac{C}{p}\bar u_p(x),\quad v_p(x)=\bar v_p(x)\int_\Omega u_p^q\leq \frac{C}{p}\bar v_p(x)\quad\text{for $p$ large},\]
implies
$$\nm{u_p}_{L^\iy(K)},~\nm{v_p}_{L^\iy(K)}\le \f{C}{p},\quad\text{for $p$ large}. $$
Thus the proof is complete.
\ep

\vskip0.23in
\section{The blow up analysis}

\subsection{Proof of Theorem \ref{thm1}}
	By Section 3, we can define
		$$\SR:=\Sigma_\mu(\delta)=\Sigma_\nu(\delta)=\{x_1,\cdots,x_k\}.$$
Then $\SR\subset\Omega$ by \eqref{3-233} and Lemma \ref{equiv-1}.
Take $r_0>0$ such that
	$$B_{4r_0}(x_i)\subset\Omega,\quad B_{2r_0}(x_i)\cap B_{2r_0}(x_j)=\emptyset,\quad\text{for}~i,j=1,\cdots,k,~i\neq j.$$
Define $x_{i,p}$ and $y_{i,p}$ by
	$$u_p(x_{i,p})=\max_{B_{2r_0}(x_i)} u_p,\quad v_p(y_{i,p})=\max_{B_{2r_0}(x_i)} v_p.$$
Up to a subsequence if necessary, we assume
\be
	u_p(x_{i,p})\to l_i,\quad v_p(y_{i,p})\to \tilde l_i,\quad\text{as}~p\to\iy.
\ee
Denote
	$$\mu_{i,p}^{-2}=pu_p^{p-1}(x_{i,p}),\quad \tilde\mu_{i,p}^{-2}=pv_p^{p-1}(y_{i,p}).$$
Then we have

\bl\lab{converge3}
	For any $i=1,\cdots,k$ and up to a subsequence of $p\to\iy$, the following hold.
	\begin{itemize}[fullwidth,itemindent=2em]
	\item[(1)]	$\mu_{i,p}\to0$ and $\tilde\mu_{i,p}\to0$.
	\item[(2)]	$x_{i,p}\to x_i$ and $y_{i,p}\to x_i$.
	\item[(3)]	$p(v_p(y_{i,p})-u_p(x_{i,p}))\to C_i$ for some $C_i\ge0$.
	\item[(4)]	$l_i=\tilde l_i\ge1$.
	\item[(5)]	Define
					$$\bcs
						w_{i,p}(x)=\f{p}{u_p(x_{i,p})}(u_p(x_{i,p}+\mu_{i,p}x)-u_p(x_{i,p})),\\
						z_{i,p}(x)=\f{p}{u_p(x_{i,p})}(v_p(x_{i,p}+\mu_{i,p}x)-u_p(x_{i,p})),\ecs
					\quad x\in\Omega_{i,p}:=\f{\Omega-x_{i,p}}{\mu_{i,p}}.$$
				then $(w_{i,p},z_{i,p})\to (U_{i,\q},V_{i,\q})$ in $\CR_{loc}^2(\R^2)$,
				where $(U_{i,\q}+\q\log l_i,V_{i,\q})$ is a solution of \eqref{sys2} and
					$$U_{i,\q}(x)=\log\f{1}{(1+\f{1}{8}l_i^\q|x|^2)^2},\quad
					  V_{i,\q}(x)=\log\f{ l_i^\q}{(1+\f{1}{8} l_i^\q|x|^2)^2},$$
				with $\q$ defined in \eqref{para}.
	\end{itemize}
\el
\bp
\begin{itemize}[fullwidth,itemindent=0em]
\item[(1)]	Follow the proof of \eqref{3-20} in Lemma \ref{equiv-2}, we have
\be\lab{tem-12}
	\lim_{p\to\iy} g_p(x_{i,p})=\lim_{p\to\iy}\max_{B_{2r_0}(x_i)}g_p=+\iy,\quad \lim_{p\to\iy} f_p(y_{i,p})=+\iy.
\ee
Since
\[g_p(x_{i,p})=\frac{u_p(x_{i,p})^{q}}{\int_{\Omega}u_p^q}\leq Cpu_p(x_{i,p})^{q}\leq Cpu_p(x_{i,p})^{p-1}=C\mu_{i,p}^{-2},\]
we obtain $\mu_{i,p}\to0$ and similarly $\tilde\mu_{i,p}\to0$.

\item[(2)]	Assume by contradiction that up to a subsequence, $x_{i,p}\to \tilde x_i$ with $\tilde x_i\in \overline{ B_{2r}(x_i)}\setminus\{x_i\}$, then $\tilde x_i\not\in\SR$.
Note from \eqref{tem-12} that
	\be \label{3-22}\liminf_{p\to\iy}u_p(x_{i,p})\ge 1, \quad \liminf_{p\to\iy}v_p(y_{i,p})\ge 1.\ee
Thus the same argument as \eqref{3-19} implies $\bar u_p(x_{i,p})\to\iy$, so
it follows from Lemma \ref{equiv-1} that $\tilde x_i\in\SR$, a contradiction.
So $x_{i,p}\to  x_i$. Similarly we can prove $y_{i,p}\to x_i$.

\item[(3)\&(5)] For any $R>0$ and $x\in B_R(0)$, we have
	$$x_{i,p}+\mu_{i,p}x\in B_{\mu_{i,p}R}(x_{i,p})\subset B_{2r_0}(x_i)\quad\text{for $p$ large.}$$
Then we have
 	$$w_{i,p}(0)=0=\max_{B_R(0)} w_{i,p}\quad\text{and}\quad\max_{B_R(0)} z_{i,p}\leq \frac{p(v_p(y_{i,p})-u_p(x_{i,p}))}{u_p(x_{i,p})}.$$
Using this fact, we can prove $(3)$ and $(5)$ exactly as Lemmas \ref{converge1} and \ref{compare1}.

\item[(4)]	This follows from \eqref{3-22} and the statment $(3)$.
\end{itemize}
The proof is complete.
\ep

\br\lab{converge4}
	Like Remark \ref{converge2}, we can scale $u_p,v_p$ at $y_{i,p}$.
Denote the scaling of $u_p,v_p$ at $y_{i,p}$ by
\be
	\begin{aligned}
	\tilde w_{i,p}(x)&=\f{p}{v_p(y_{i,p})}\sbr{u_p(y_{i,p}+\tilde \mu_{i,p}x)-v_p(y_{i,p})},\\
	\tilde z_{i,p}(x)&=\f{p}{v_p(y_{i,p})}\sbr{v_p(y_{i,p}+\tilde \mu_{i,p}x)-v_p(y_{i,p})},
	\end{aligned}\quad\quad x\in\tilde\Omega_{i,p}:=\f{\Omega-y_{i,p}}{\tilde\mu_{i,p}}.
\ee
We see that $\tilde\Omega_{i,p}\to\R^2$ as $p\to\iy$.
It is easy to check that
	$$1+\f{\tilde w_{i,p}(x)}{p}=\f{u_p(y_{i,p}+\tilde\mu_{i,p}x)}{v_p(y_{i,p})},
		\quad 1+\f{\tilde z_{i,p}(x)}{p}=\f{v_p(y_{i,p}+\tilde\mu_{i,p}x)}{v_p(y_{i,p})},$$
and
\be
	\bcs
		-\Delta\tilde w_{i,p}=(1+\f{\tilde z_{i,p}}{p})^p, \quad\text{in}~\tilde\Omega_{i,p},\\
		-\Delta\tilde z_{i,p}=v_p(y_{i,p})^{\q_p}(1+\f{\tilde w_{i,p}}{p})^q, \quad\text{in}~\tilde\Omega_{i,p},\\
		\tilde w_{i,p}=\tilde z_{i,p}=-p,\quad\text{on}~\pa\tilde\Omega_{i,p}.
	\ecs
\ee
Moreover, $\tilde z_{i,p}(0)=0=\max_{B_R(0)}\tilde z_{i,p}$ for any $R>0$.
Following the proof of Lemmas \ref{converge1} and \ref{compare1}, we have that
		$$\tilde w_{i,p}\to \tilde U_{i,\q},\quad \tilde z_{i,p}\to \tilde V_{i,\q},
			\quad\text{in}~\CR_{loc}^2(\R^2),\quad\text{as}~p\to\iy,$$
where $(\tilde U_{i,\q}+\q\log\tilde l_i,\tilde V_{i,\q})$ is a solution of \eqref{sys2} and
	\be
		\tilde U_{i,\q}(x)=\log\f{\tilde l_i^{-\q}}{(1+\f{1}{8}|x|^2)^2},\quad
		\tilde V_{i,\q}(x)=\log\f{1}{(1+\f{1}{8}|x|^2)^2}.
	\ee
\er

By Lemma \ref{tem-11}, we see that
	\be\label{3-23} u_p=O(\tfrac{1}{p}), \quad v_p=O(\tfrac{1}{p}),\qquad\text{in}~\CR_{loc}(\overline{\Omega}\setminus\SR).\ee
More precisely, we have
\bl\lab{converge-2}
	Up to a subsequence, we have
		$$pu_p(x)\to 8\pi \sum_{i=1}^k l_iG(x,x_i),\quad pv_p(x)\to 8\pi \sum_{i=1}^k l_iG(x,x_i),
			\quad\text{in}~\CR^2_{loc}(\overline{\Omega}\setminus\SR),\quad\text{as}~p\to\iy.$$
\el
\bp Given any compact subset $K\Subset \overline{\Omega}\setminus \SR$,
	there is small $r>0$ such that $K\subset \overline{\Omega}\setminus \sum_{i=1}^k B_r(x_i)$. Then for any $x\in K$ and $d\in(0,r)$, it follows from the Green representation formula and \eqref{3-23} that
$$
	\begin{aligned}
		pu_p(x)
		&=p\int_\Omega G(x,y)v_p^p(y) \rd y\\
		&=\sum_{i=1}^k p\int_{B_d(x_i)} G(x,y)v_p^p(y) \rd y+o_p(1)\int_{\Omega\setminus \sum_{i=1}^k B_r(x_i)} G(x,y)\rd y\\
		&\to\sum_{i=1}^k \ga_i G(x,x_i),\quad\text{uniformly for $x\in K$ as}~p\to\iy,
	\end{aligned}
$$
where
	$$\ga_i=\lim_{d\to0}\lim_{p\to\iy}p\int_{B_d(x_i)}v_p^p(y)\rd y.$$
Again by the Green representation formula,
	$$p\nabla u_p(x)=p\int_\Omega \nabla_x G(x,y)v_p^p(y) \rd y\to \sum_{i=1}^k \ga_i\nabla_x G(x,x_i),$$
so $pu_p(x)\to \sum_{i=1}^k \ga_i G(x,x_i)$ in $\CR_{loc}^1(\overline{\Omega}\setminus \SR)$.  Similarly, we can obtain $pv_p(x)\to \sum_{i=1}^k \tilde{\ga}_i G(x,x_i)$ in $\CR_{loc}^1(\overline{\Omega}\setminus \SR)$, where
$$\tilde \ga_i=\lim_{d\to0}\lim_{p\to\iy}p\int_{B_d(x_i)}u_p^q(y)\rd y.$$
From here and $-\Delta (pu_p)=pv_p^p\to 0$, $-\Delta (pv_p)=pu_p^q\to 0$ in $\Omega\setminus \SR$, it follows from the standard elliptic estimates that
$$pu_p(x)\to \sum_{i=1}^k \ga_i G(x,x_i), \quad pv_p(x)\to \sum_{i=1}^k \tilde\ga_i G(x,x_i),\quad\text{in}~\CR_{loc}^2(\overline{\Omega}\setminus \SR),\quad\text{as}~p\to\iy.$$

It remains to prove $\gamma_i=\tilde\gamma_i=8\pi l_i$.	For every $i=1,\cdots,k$, since
$$
	\begin{aligned}
	\lim_{p\to\iy}p\int_{B_d(x_i)}v_p^p
	&\ge \lim_{p\to\iy}p\int_{B_{d/2}(x_{i,p})}v_p^p 	
			=\lim_{p\to\iy} u_p(x_{i,p})\int_{B_{\f{d}{2\mu_{i,p}}}(0)}\Big(1+\f{z_{i,p}}{p}\Big)^p\\
	&\ge l_i\int_{\R^2}e^{V_{i,\q}}=8\pi l_i,
	\end{aligned}
$$
and
$$
	\begin{aligned}
	\lim_{p\to\iy}p\int_{B_d(x_i)}u_p^q
	&\ge \lim_{p\to\iy}p\int_{B_{d/2}(x_{i,p})}u_p^q	
			=\lim_{p\to\iy} u_p(x_{i,p})^{\q_p+1}\int_{B_{\f{d}{2\mu_{i,p}}}(0)}\Big(1+\f{w_{i,p}}{p}\Big)^q\\
	&\ge l_i^{\q+1}\int_{\R^2}e^{U_{i,\q}}=8\pi l_i,
	\end{aligned}
$$
we get $\ga_i,\tilde\ga_i\ge 8\pi l_i$.

Noting that for $h=1,\cdots,k$ and $x\in B_r(x_h)\setminus\{x_h\}$, we have
\be\lab{tem-3}
	p\nabla u_p(x)\to \sum_{i=1}^k \ga_i\nabla_x G(x,x_i)=-\f{\ga_h}{2\pi}\f{x-x_h}{|x-x_h|^2}+O(1),
\ee
\be\lab{tem-4}
	p\nabla v_p(x)\to \sum_{i=1}^k \tilde\ga_i\nabla_x G(x,x_i)=-\f{\tilde\ga_h}{2\pi}\f{x-x_h}{|x-x_h|^2}+O(1).
\ee
Applying the Pohozaev identity \eqref{pho1} with $y=x_h$, $\Omega'=B_d(x_h)$, $(u,v)=(u_p,v_p)$, we obtain
\be\lab{tem-5}
	\begin{aligned}
			&\quad~~ \f{2p^2}{p+1}\int_{B_d(x_h)}v_p^{p+1}\rd x+\f{2p^2}{q+1}\int_{B_d(x_h)}u_p^{q+1}\rd x\\
			&=-d\int_{\pa B_d(x_h)}\abr{p\nabla u_p,p\nabla v_p}\rd \sigma_x
				+2d\int_{\pa B_d(x_h)}\abr{p\nabla u_p,\nu}\abr{p\nabla v_p,\nu}\rd\sigma_x\\
			&\quad	+d\int_{\pa B_d(x_h)} \f{p^2v_p^{p+1}}{p+1}+\f{p^2u_p^{q+1}}{q+1}\rd\sigma_x.
	\end{aligned}
\ee
Using \eqref{tem-3}-\eqref{tem-4} and \eqref{3-23}, we obtain
	$$\text{Right hand side of \eqref{tem-5}}=\f{\ga_h\tilde\ga_h}{2\pi}+O(d)+o_p(1).$$
Meanwhile, since Lemma \ref{converge3} shows $\max_{B_d(x_h)}v_p\to l_h$ and $\max_{B_d(x_h)}u_p\to l_h$ as $p\to \infty$, we have
	$$\begin{aligned}
		\text{Left hand side of \eqref{tem-5}}
		&\le 2l_h\lim_{p\to\iy}\sbr{p\int_{B_d(x_h)}(u_p^q+v_p^p)} +o_p(1)\\
		&=2l_h(\ga_h+\tilde\ga_h)+o_d(1)+o_p(1).
	\end{aligned}$$
So letting $p\to\iy$ and $d\to0$ in \eqref{tem-5}, we get $\f{\ga_h\tilde\ga_h}{2\pi}\leq 2l_h(\ga_h+\tilde\ga_h)$ or equivalently,
	$$\f{1}{\ga_h}+\f{1}{\tilde\ga_h}\ge\f{1}{4\pi l_h}.$$
From here and $\ga_h,\tilde\ga_h\ge 8\pi l_h$, we conclude that $\ga_h=\tilde\ga_h=8\pi l_h$.
\ep

	To proceed, we give a decay estimate of $w_{i,p},z_{i,p}$, which will be used to apply the Dominated Convergence Theorem.
\bl\lab{decay}
	For any $\ga\in(0,4)$, there exist small $r_\ga>0$, large $R_\ga>0$, $p_\ga>1$ and $C_\ga>0$ such that
	\be
		w_{i,p}(x),~z_{i,p}(x)\le \ga\log\f{1}{|x|}+C_\ga,
	\ee
	for any $2R_\ga\le |x|\le \f{r_\ga}{\mu_{i,p}}$, $p\ge p_\ga$ and $i=1,\cdots,k$.
\el
\bp
	First we prove it for $w_{i,p}$.
For simiplicity, we omit the parameter $i$ and denote
	$$w_p=w_{i,p}, ~z_p=z_{i,p},~~\mu_p=\mu_{i,p}, ~x_p=x_{i,p}, ~\Omega_p=\Omega_{i,p}.$$
By the Green Representation formula we have for any $x\in\Omega_p$,
$$
	\begin{aligned}
		u_p(x_p+\mu_px)
		&=\int_\Omega G(x_p+\mu_px,y)v_p^p(y)\rd y\\
		&=\f{u_p(x_p)}{p}\int_{\Omega_p} G(x_p+\mu_px,x_p+\mu_pz)\sbr{1+\f{z_p(z)}{p}}^p\rd z.
	\end{aligned}
$$
Then
	$$w_p(x)=-p+\int_{\Omega_p} G(x_p+\mu_px,x_p+\mu_pz)\sbr{1+\f{z_p(z)}{p}}^p\rd z.$$
Since $w_p(0)=0$ and $G(x,y)=-\f{1}{2\pi}\log|x-y|-H(x,y)$, we have
$$
	\begin{aligned}
		w_p(x)
		&=w_p(x)-w_p(0)\\
		&=\int_{\Omega_p} \mbr{G(x_p+\mu_px,x_p+\mu_pz)-G(x_p,x_p+\mu_pz)}\sbr{1+\f{z_p(z)}{p}}^p\rd z\\
		&=\f{1}{2\pi}\int_{\Omega_p} \log\f{|z|}{|z-x|}\sbr{1+\f{z_p(z)}{p}}^p\rd z\\
			&\quad -\int_{\Omega_p} \mbr{H(x_p+\mu_px,x_p+\mu_pz)-H(x_p,x_p+\mu_pz)}\sbr{1+\f{z_p(z)}{p}}^p\rd z\\
		&=:\Rmnum{1}+\Rmnum{2}.
	\end{aligned}
$$
Since $H$ is smooth, we have $\Rmnum{2}=O(1)$ if $|\mu_px|\le r_\ga$ with $r_\ga$ small to be chosen later.

Let $\e=\f{2\pi}{3}(4-\ga)>0$ and take $R_\ga>0$ such that $\int_{B_{R_\ga}(0)}e^{V_{i,\q}}>\int_{\R^2}e^{V_{i,\q}}-\f{\e}{2}=8\pi-\f{\e}{2}$, where $V_{i,\q}$ is given in Lemma \ref{converge3}.
Then for $p$ large enough,
	\be\label{3-24}\int_{B_{R_\ga}(0)}\sbr{1+\f{z_p}{p}}^p\ge \int_{B_{R_\ga}(0)}e^{V_{i,\q}}-\f{\e}{2}> 8\pi-\e.\ee
Next, we divide the integral $\Rmnum{1}$ in the following way
$$
	\begin{aligned}
		\Rmnum{1}
		&=\f{1}{2\pi}\int_{\lbr{|z|\le R_\ga}} \log\f{|z|}{|z-x|}\sbr{1+\f{z_p}{p}}^p\rd z\\
		&\quad	+\f{1}{2\pi}\int_{\lbr{|z|\ge \f{2r_\ga}{\mu_p}} }\log\f{|z|}{|z-x|}\sbr{1+\f{z_p}{p}}^p\rd z\\
		&\quad	+\f{1}{2\pi}\int_{\lbr{R_\ga\le |z|\le \f{2r_\ga}{\mu_p},~|z|\le 2|z-x|}} \log\f{|z|}{|z-x|}\sbr{1+\f{z_p}{p}}^p\rd z\\
		&\quad	+\f{1}{2\pi}\int_{\lbr{R_\ga\le |z|\le \f{2r_\ga}{\mu_p},~|z|\ge 2|z-x|}} \log|z|\sbr{1+\f{z_p}{p}}^p\rd z\\
		&\quad	+\f{1}{2\pi}\int_{\lbr{R_\ga\le |z|\le \f{2r_\ga}{\mu_p},~|z|\ge 2|z-x|}} \log\f{1}{|z-x|}\sbr{1+\f{z_p}{p}}^p\rd z\\
		&=:\Rmnum{1}_1+\Rmnum{1}_2+\Rmnum{1}_3+\Rmnum{1}_4+\Rmnum{1}_5.
	\end{aligned}
$$
Let $2R_\ga\le |x|\le \f{r_\ga}{\mu_p}$.
We observe that if $|z|\le R_\ga$, then $|x|\ge 2|z|$, $|x-z|\geq \frac12 |x|$ and hence
	$$\log\f{|z|}{|z-x|}\le \log\f{2R_\ga}{|x|}\le 0.$$
From here and \eqref{3-24}, we have
	$$\Rmnum{1}_1\le \f{1}{2\pi}\log\f{2R_\ga}{|x|}\int_{\lbr{|z|\le R_\ga}} \sbr{1+\f{z_p}{p}}^p
	\le C+(4-\f{\e}{2\pi})\log\f{1}{|x|}.$$
If $|z|\ge \f{2r_\ga}{\mu_p}$, then $|z|\ge2|x|$ and hence
	$$\log\f{2}{3}\le\log\f{|z|}{|z-x|}\le\log2,$$
which implies
	$$\Rmnum{1}_2\le C\int_{\Omega_p}\sbr{1+\f{z_p}{p}}^p= \frac{Cp}{u_p(x_p)}\int_{\Omega}v_p^p\le C.$$
Similarly,
	$$\Rmnum{1}_3\le C\int_{\Omega_p}\sbr{1+\f{z_p}{p}}^p\le C.$$
From the proof of Lemma \ref{converge-2} and $u_p(x_p)\to l_i$, we see that
	$$\lim_{r\to0}\lim_{p\to\iy}\int_{\lbr{|z|\le\f{2r}{\mu_p}}}\sbr{1+\f{z_p}{p}}^p
		=\lim_{r\to0}\lim_{p\to\iy}\f{p}{u_p(x_p)}\int_{B_{2r}(x_p)}v_p^p=8\pi.$$
Thus we can choose $r_\ga>0$ small such that for $p$ large,
	$$\int_{\lbr{|z|\le\f{2r_\ga}{\mu_p}}}\sbr{1+\f{z_p}{p}}^p\le 8\pi+\e$$
and then \eqref{3-24} yields
$$\int_{\lbr{R_\ga\leq |z|\le\f{2r_\ga}{\mu_p}}}\sbr{1+\f{z_p}{p}}^p\le 2\e.$$
We also observe that $|z|\ge2|z-x|$ implies $|z|\le2|x|$,
so
	$$\Rmnum{1}_4\le \f{1}{2\pi}\log2|x|\int_{\lbr{R_\ga\le|z|\le\f{2r_\ga}{\mu_p}}}\sbr{1+\f{z_p}{p}}^p
		\le C-\f{\e}{\pi}\log\f{1}{|x|}.$$
Finally, recall Lemma \ref{converge3} that
\[\max z_p\leq \frac{p(v_p(y_p)-u_p(x_p))}{u_p(x_p)}\to \frac{C_i}{l_i},\quad\text{as $p\to \infty$},\]
which implies $\max z_p\leq C$ for $p$ large.
Since $|z-x|\ge1$ implies $\log\f{1}{|z-x|}\le0$, we have
	\begin{align*}\Rmnum{1}_5&\le \f{1}{2\pi}\int_{\lbr{R_\ga\le |z|\le \f{2r_\ga}{\mu_p},~|z|\ge 2|z-x|, |z-x|\leq 1}} \log\f{1}{|z-x|}\sbr{1+\f{z_p}{p}}^p\rd z\\
&\le C\sbr{1+\f{C}{p}}^p\int_{\lbr{|z-x|\le 1}}\log\f{1}{|z-x|}\rd z\le C.\end{align*}
As a result, we get
	$$\Rmnum{1}=\sum_{j=1}^5 \Rmnum{1}_j\le C_\ga+(4-\f{3\e}{2\pi})\log\f{1}{|x|}=\ga\log\f{1}{|x|}+C_\ga.$$
This proves $w_p(x)\le \ga\log\f{1}{|x|}+C_\ga$ for $2R_\ga\le |x|\le \f{r_\ga}{\mu_p}$ and $p$ large.

	Using the following facts $z_{i,p}(0)\to V_{i,\q}(0)$,
	$$\lim_{R\to+\iy}\lim_{p\to\iy}u_p^{\q_p}(x_{i,p})\int_{B_R(0)}\sbr{1+\f{w_{i,p}}{p}}^q
		\ge l_i^\q\int_{\R^2}e^{U_{i,\q}}=8\pi,$$
and
	$$\lim_{r\to0}\lim_{p\to\iy}u_p^{\q_p}(x_{i,p})\int_{\lbr{|z|\le\f{2r}{\mu_{i,p}}}}\sbr{1+\f{w_{i,p}}{p}}^q
		=\lim_{r\to0}\lim_{p\to\iy}\f{p}{u_p(x_{i,p})}\int_{B_{2r}(x_{i,p})}u_p^q=8\pi,$$
we can prove a similar result for $z_{i,p}$.
\ep

\br\lab{decay-1}
	For any $0<\ga<4$ and for $p$ large, if $2R_\ga\le |x|\le \f{r_\ga}{\mu_{i,p}}$, then Lemma \ref{decay} implies
		$$\sbr{1+\f{w_{i,p}(x)}{p}}^q=e^{q\log(1+\f{w_{i,p}(x)}{p})}\le e^{\f{q}{p}w_{i,p}(x)}
			\le \f{C_\ga}{|x|^\ga}.$$
	Meanwhile, by the convergence of $w_{i,p}$ in $\CR_{loc}^2(\R^2)$, we have $\sbr{1+\f{w_{i,p}(x)}{p}}^q\le C$ for $|x|\le 2R_\ga$ and $p$ large.
	Therefore, for any $|x|\leq \f{r_\ga}{\mu_{i,p}}$ and $p$ large, we have
		$$0\le \sbr{1+\f{w_{i,p}(x)}{p}}^q\le \f{C_\ga}{1+|x|^\ga},$$
	and similarly,
		$$0\le\sbr{1+\f{z_{i,p}(x)}{p}}^p\le \f{C_\ga}{1+|x|^\ga}.$$
\er

	As a direct application of the above decay estimates, we prove

\bl\lab{tem-13}
	It holds
		$$l_i=\sqrt e,\quad\text{for any}~i=1,\cdots,k.$$
\el
\bp
By Remark \ref{decay-1}, there is $r>0$ and $p_0>0$ such that for any $|x|\le \f{r}{\mu_{i,p}}$ and $p\geq p_0$, we have
	\be\label{33-25} 0\le\sbr{1+\f{z_{i,p}(x)}{p}}^p, \;\sbr{1+\f{w_{i,p}(x)}{p}}^q\le \f{C}{1+|x|^3}.\ee
Since $|G(x_{i,p},y)|\leq C|\log |x_{i,p}-y||\leq C$ for any $y\in \Omega\setminus B_r(x_{i,p})$, we have
$$\left|\int_{\Omega\setminus B_r(x_{i,p})} G(x_{i,p},y)v_p^p(y)\rd y\right|
\leq C\int_{\Omega}v_p^p(y)\rd y\leq \frac{C}{p}.$$
Then by the Green representation formula, we have
$$\begin{aligned}
	u_p(x_{i,p})
	&=\int_\Omega G(x_{i,p},y)v_p^p(y)\rd y\\
	&=\int_{B_r(x_{i,p})}G(x_{i,p},y)v_p^p(y)\rd y+o_p(1)\\ &=\f{u_p(x_{i,p})}{p}\int_{B_{\f{r}{\mu_{i,p}}}(0)}G(x_{i,p},x_{i,p}+\mu_{i,p}z)\sbr{1+\f{z_{i,p}(z)}{p}}^p\rd z+o_p(1),
\end{aligned}$$
so
\[\lim_{p\to\infty}\f{1}{p}\int_{B_{\f{r}{\mu_{i,p}}}(0)}G(x_{i,p},x_{i,p}+\mu_{i,p}z)\sbr{1+\f{z_{i,p}(z)}{p}}^p\rd z=1.\]
On the other hand,
by \eqref{33-25}, Lemma \ref{converge3} and the Dominated Convergence Theorem, we get
\[\lim_{p\to\infty}\int_{B_{\f{r}{\mu_{i,p}}}(0)} \sbr{1+\f{z_{i,p}(z)}{p}}^p\rd z
=\int_{\R^2}e^{V_{i,\theta}}=8\pi,\]
\begin{align*}&\lim_{p\to\infty}\int_{B_{\f{r}{\mu_{i,p}}}(0)} \sbr{\f{1}{2\pi}\log|z|
			+H(x_{i,p},x_{i,p}+\mu_{i,p}z)}\sbr{1+\f{z_{i,p}(z)}{p}}^p\rd z\\
&\qquad=\int_{\R^2} \sbr{\f{1}{2\pi}\log|z|
			+H(x_{i},x_{i})}e^{V_{i,\theta}(z)}\rd z=C <\infty.\end{align*}
From here, $G(x,y)=-\f{1}{2\pi}\log|x-y|-H(x,y)$ and $\mu_{i,p}^{-2}=pu_p(x_{i,p})^{p-1}$, we have
$$\begin{aligned}
	&\quad \f{1}{p}\int_{B_{\f{r}{\mu_{i,p}}}(0)}  G(x_{i,p},x_{i,p}+\mu_{i,p}z)\sbr{1+\f{z_{i,p}(z)}{p}}^p\rd z\\
	&=-\f{1}{2\pi}\f{\log\mu_{i,p}}{p}\int_{B_{\f{r}{\mu_{i,p}}}(0)} \sbr{1+\f{z_{i,p}(z)}{p}}^p\rd z\\
	&\quad -\f{1}{p}\int_{B_{\f{r}{\mu_{i,p}}}(0)} \sbr{\f{1}{2\pi}\log|z|
			+H(x_{i,p},x_{i,p}+\mu_{i,p}z)}\sbr{1+\f{z_{i,p}(z)}{p}}^p\rd z\\
	&=\f{1}{4\pi}\sbr{\f{\log p}{p}+\f{p-1}{p}\log u_p(x_{i,p})}(8\pi+o_p(1))+o_p(1)\\
	&=2\log l_i+o_p(1).
\end{aligned}$$
Thus $2\log l_i=1$, i.e., $l_i=\sqrt e$.
\ep

\bl
	It holds
		$$p\int_\Omega|\nabla u_p|^2\to k8\pi e,\quad p\int_\Omega|\nabla v_p|^2\to k8\pi e,
			\quad p\int_\Omega\nabla u_p\cdot\nabla v_p\to k8\pi e,$$
	as $p\to\iy$.
\el
\bp
	By \eqref{3-23}, \eqref{33-25} and the Dominated Convergence Theorem, we have
$$
	\begin{aligned}
		p\int_\Omega|\nabla u_p|^2
		&=p\int_\Omega v_p^pu_p=\sum_{i=1}^k p\int_{B_r(x_{i,p})}v_p^pu_p+o_p(1)\\
		&=\sum_{i=1}^k u_p^2(x_{i,p})\int_{B_{\f{r}{\mu_{i,p}}}(0)}\sbr{1+\f{z_{i,p}}{p}}^p\sbr{1+\f{w_{i,p}}{p}}+o_p(1)\\
		&\to \sum_{i=1}^k l_i^2\int_{\R^2}e^{V_{i,\q}}=k8\pi e,\quad\text{as}~p\to\iy.
	\end{aligned}
$$
The other assertions can be proved similarly.
\ep

\bl
For any $i=1,\cdots,k$, we have
		$$\f{x_{i,p}-y_{i,p}}{\mu_{i,p}}\to0,\quad\text{as}~p\to\iy.$$
\el
\bp
	Suppose that up to a subsequence, $\f{|x_{i,p}-y_{i,p}|}{\mu_{i,p}}\to +\iy$ as $p\to\iy$.
Let $\tilde \mu_{i,p}$, $\tilde z_{i,p}$ be defined as in Remark \ref{converge4}.
Then it follows from Lemma \ref{converge3} that
	$$\f{\mu_{i,p}^2}{\tilde \mu_{i,p}^2}=\sbr{\f{v_p(y_{i,p})}{u_p(x_{i,p})}}^{p-1}=\Big(1+\f{C_i/l_i+o_p(1)}{p}\Big)^{p-1}\to e^{C_i/l_i}\in[1,+\iy).$$
Thus for any $R>0$, we have $B_{\mu_{i,p}R}(x_{i,p})\cap B_{\tilde \mu_{i,p}R}(y_{i,p})=\emptyset$ for $p$ large, therefore,
$$
	\begin{aligned}
		&\quad k8\pi e+o_p(1)=p\int_\Omega\nabla u_p\cdot\nabla v_p= p\int_\Omega v_p^{p+1}\\
		&\ge \sum_{j=1,j\neq i}^k p\int_{B_{\mu_{j,p}R}(x_{j,p})} v_p^{p+1}
			+p\int_{B_{\mu_{i,p}R}(x_{i,p})} v_p^{p+1}+p\int_{B_{\tilde\mu_{i,p}R}(y_{i,p})} v_p^{p+1}\\
		&= \sum_{j=1}^k u_p^2(x_{j,p})\int_{B_R(0)} \sbr{1+\f{z_{j,p}}{p}}^{p+1}
			+v_p^2(y_{i,p})\int_{B_R(0)}\sbr{1+\f{\tilde z_{i,p}}{p}}^{p+1}\\
		&\to e\sum_{j=1}^k \int_{B_R(0)} e^{V_{j,\q}}+e\int_{B_R(0)} e^{\tilde V_{i,\q}},\quad\text{as}~p\to\iy.
	\end{aligned}
$$
Since $\int_{\R^2} e^{V_{j,\q}}=\int_{\R^2}e^{\tilde V_{i,\q}}=8\pi$, taking $R>0$ large enough
we obtain $k8\pi e\ge (k+\f{1}{2})8\pi e$, a contradiction.
Thus $\f{y_{i,p}-x_{i,p}}{\mu_{i,p}}$ are uniformly bounded and up to a subsequence, $\f{y_{i,p}-x_{i,p}}{\mu_{i,p}}\to a\in\R^2$ as $p\to\iy$.
Since $z_{i,p}(\f{y_{i,p}-x_{i,p}}{\mu_{i,p}})=\max z_{i,p}$, we have $V_{i,\q}(a)=\max V_{i,\q}=V_{i,\q}(0)$ and so $a=0$.
\ep

	Finally we compute the location of the blow up points $x_1,\cdots,x_k$.
\bl
	It holds
		$$\nabla R(x_i)-2\sum_{j=1,j\neq i}^k\nabla G(x_i,x_j)=0,$$
	for every $i=1,\cdots,k$.
\el
\bp
	We only prove the conclusion for $i=1$.
Applying the Pohozaev identity \eqref{pho2} with $\Omega'=B_r(x_1)$ and $(u,v)=(u_p,v_p)$, we obtain from \eqref{3-23} that
\be\lab{tem-6}
	\begin{aligned}
			&-\int_{\pa B_r(x_1)}\abr{p\nabla u_p,\nu}p\pa_iv_p+\abr{p\nabla v_p,\nu}p\pa_iu_p\rd\sigma_x
			+\int_{\pa B_r(x_1)}\abr{p\nabla u_p,p\nabla v_p}\nu_i\rd \sigma_x\\
			&=\int_{\pa B_r(x_1)}\sbr{\f{p^2v_p^{p+1}}{p+1}+\f{p^2u_p^{q+1}}{q+1}}\nu_i\rd\sigma_x=o_p(1),\quad i=1,2.
	\end{aligned}
\ee	
Set
	$$h(x)=\sum_{i=1}^k G(x,x_i),$$
then it follows from Lemmas \ref{converge-2} and \ref{tem-13} that
	$$p\nabla u_p(x)\to 8\pi\sqrt e\nabla h(x),
		\quad p\nabla v_p(x)\to 8\pi\sqrt e\nabla h(x),$$
uniformly for $x\in \pa B_r(x_1)$ as $p\to\iy$.
Thus we deduce from \eqref{tem-6} that
\be\lab{tem-7}
	\int_{\pa B_r(x_1)}\abr{\nabla h,\nu}\pa_i h-\f{1}{2}|\nabla h|^2\nu_i\rd\sigma=0,\quad i=1,2.
\ee
Denote $\bar h(x)=H(x,x_1)-\sum_{j=2}^k G(x,x_j)$, we have $\bar h\in\CR^\iy (B_{r}(x_1))$ and $h(x)=-\frac1{2\pi}\log|x-x_1|-\bar h(x)$.
For any $x\in\pa B_r(x_1)$, by direct computations we get
	$$\abr{\nabla h(x),\nu}\pa_i h(x)-\f{1}{2}|\nabla h(x)|^2\nu_i
		=\f{(x-x_1)_i}{8\pi^2r^3}+\f{1}{2\pi r}\pa_i \bar h(x)+O(1).$$
Inserting this into \eqref{tem-7}, we see that $\pa_i \bar h(\xi_{r,i})=O(r)$ for some $\xi_{r,i} \to x_1$ as $r\to0$.
So letting $r\to0$, we have $\nabla \bar h (x_1)=0$. This, together with $R(x)=H(x,x)$, implies
	$$\nabla R(x_1)-2\sum_{j=2}^k\nabla G(x_1,x_j)=0.$$
This completes the proof.
\ep

\bp[Proof of Theorem \ref{thm1}]
	The proof is a combination of the above results.
\ep

\subsection{The least energy solution}\lab{leastenergy}
	
	To introduce the least energy solution in \cite{exist3}, we need some notations.
Consider the energy functional
	$$I_p(u,v)=\int_\Omega\nabla u\cdot\nabla v-\int_\Omega\f{|v|^{p+1}}{p+1}+\f{|u|^{q+1}}{q+1},\quad\text{for}~u,v\in H_0^1(\Omega),$$
and
	$$J_p(u)=\f{p}{p+1}\int_\Omega |\Delta u|^{1+\f{1}{p}}-\f{1}{q+1}\int_\Omega |u|^{q+1}$$
defined on $E=W_0^{1,1+\f{1}{p}}(\Omega)\cap W^{2,1+\f{1}{p}}(\Omega)$, endowed with  the norm
	$$\nm{u}_E^{1+\f{1}{p}}=\int_\Omega |\Delta u|^{1+\f{1}{p}}.$$
Then $I_p\in\CR^1(H_0^1(\Omega)\times H_0^1(\Omega))$ and $J_p\in\CR^1(E)$.
Clearly $(u,v)$ is a solution of \eqref{sys1} if and only if $(u,v)$ is a critical point of $I_p$,
if and only if $u$ is a critical point of $J_p$ and $v=(-\Delta u)^{\f{1}{p}}$.
A solution $(u_p,v_p)$ of \eqref{sys1} is called {\it a least energy solution}  if $u_p$ is a critical point of $J_p$
with the least energy among all nontrivial critical points of $J_p$.
It was proved in \cite{exist3} that \eqref{sys1} has a least energy solution $(u_p,v_p)$
such that $v_p=(-\Delta u_p)^{\f{1}{p}}$ and  $J_p(u_p)=\inf_{u\in\NR_p} J_p(u)$ with
	$$\NR_p:=\lbr{u\in E\setminus\{0\}:~J_p'(u)[u]=\int_\Omega |\Delta u|^{1+\f{1}{p}}-\int_\Omega u^{q+1}=0 }.$$
Here we prove that this least energy solution satisfy \eqref{con1} with $\beta=8\pi e$.
\bl\lab{tem-14}
	Let $p,q$ satisfy \eqref{para}. Then for a sequence of least energy solutions $(u_p,v_p)$ defined as above, it holds
	\be\label{4-1}\limsup_{p\to+\infty}p\int_\Omega\nabla u_p\cdot\nabla v_p\le 8\pi e.\ee
\el
\bp
	We construct a proper test function. Without loss of generality, we assume $0\in\Omega$.
Let
	$$\vp_p(x)=\sqrt e\sbr{ 1+\f{z(y)}{p} }, \quad\text{with}~y=\f{x}{\e_p},$$
where $z(y)=-2\log(1+\f{1}{8}e^{\f{\q}{2}}|y|^2)$ and $\e_p^{-2}=p e^{\f{p-1}{2}}$.
Take $r>0$ small such that $B_{2r}(0)\subset\Omega$.
Take a cut-off function $\eta\in\CR^\iy_0(B_{2r}(0))$ such that
	$$0\le \eta\le 1,~~ \eta(x)=1~\text{for}~|x|\le r,$$
and
	$$|\nabla\eta(x)|\le \f{C}{r},\quad
		|\Delta\eta(x)|\le \f{C}{r^2},\quad~\text{for}~r\le |x|\le 2r. $$
Let $\psi_p=\vp_p\eta$. Then
\[\int_\Omega |\Delta \psi_p|^{1+\f{1}{p}}\rd x
		=\int_{|x|\le r} |\Delta \vp_p|^{1+\f{1}{p}}\rd x
			+\int_{r\le|x|\le 2r} |\Delta \vp_p\eta+2\nabla \vp_p\nabla\eta+\vp_p\Delta\eta|^{1+\f{1}{p}}\rd x=:I_1+I_2.\]
By $-\Delta z(y)=e^{z(y)+\frac{\theta}{2}}$ and $\int_{\R^2}e^{z(y)+\frac{\theta}{2}}=8\pi$, we have
$$
	\begin{aligned}
		I_1=&\int_{|x|\le r} |\Delta \vp_p|^{1+\f{1}{p}}\rd x
=\f{e^{\frac{p+1}{2p}}}{p^{\frac{p+1}{p}}\e_p^{2/p}}\int_{|y|\le \e_p^{-1}r} |\Delta z(y)|^{1+\f{1}{p}}\rd y\\
=&\frac{e}{p}\int_{|y|\le \e_p^{-1}r} e^{(z(y)+\f{\q}{2})(1+\f{1}{p})}\rd y=\f{8\pi e}{p}(1+o_p(1)).
	\end{aligned}
$$
On the other hand, by \[(a+b+c)^{1+\frac1p}\leq 2(a^{1+\frac1p}+b^{1+\frac1p}+c^{1+\frac1p}), \quad\forall a,b,c\geq 0,\; p\geq 2, \]
we have
\[I_2\leq C\left[\int_{r\le|x|\le 2r} |\Delta \vp_p|^{1+\f{1}{p}}+
\frac{1}{r^{1+\frac1p}}\int_{r\le|x|\le 2r} |\nabla \vp_p|^{1+\f{1}{p}}+
\frac{1}{r^{2+\frac2p}}\int_{r\le|x|\le 2r} |\vp_p|^{1+\f{1}{p}}\right].\]
By direct computations, we have
\[\int_{r\le|x|\le 2r} |\Delta \vp_p|^{1+\f{1}{p}}=\frac{e}{p}\int_{\e_p^{-1}r\leq |y|\le 2\e_p^{-1}r} e^{(z(y)+\f{\q}{2})(1+\f{1}{p})}\rd y=\frac{o_p(1)}{p},\]
and similarly,
$$\int_{r\le|x|\le 2r} |\nabla \vp_p|^{1+\f{1}{p}}=\frac{o_p(1)}{p},\quad \int_{r\le|x|\le 2r} |\vp_p|^{1+\f{1}{p}}=\frac{o_p(1)}{p}.$$
Therefore,
$$\int_\Omega |\Delta \psi_p|^{1+\f{1}{p}}\rd x=\f{8\pi e}{p}(1+o_p(1)).$$
Furthermore,
$$
	\begin{aligned}
	\int_\Omega \psi_p^{q+1}\rd x
	&=e^{\f{q+1}{2}}\e_p^2 \int_{|y|\le 2\e_p^{-1}r} \eta^{q+1}(\e_p y) \sbr{1+\f{z(y)}{p}}^{q+1}\rd y\\
	&=\f{e^{\f{\q_p}{2}+1}}{p} \left(\int_{|y|\le \e_p^{-1}r} e^{z(y)}(1+o_p(1))\rd y+o_p(1)\right)\\
	&=\f{8\pi e}{p}(1+o_p(1)),
	\end{aligned}
$$
so $t_p\psi_p\in\NR_p$ with
	$$t_p=\sbr{\f{\int_\Omega |\Delta \psi_p|^{1+\f{1}{p}}}{\int_\Omega \psi_p^{q+1}}}^{\f{p}{pq-1}}=1+o_p(1).$$
Since $u_p$ is a minimzier of $J_p|_{\NR_p}$ and for any $u\in\NR_p$,
	$$J_p(u)=\sbr{\f{p}{p+1}-\f{1}{q+1}}\int_\Omega |\Delta u|^{1+\f{1}{p}},$$
we have
	$$\int_\Omega\nabla u_p\cdot\nabla v_p=\int_\Omega |\Delta u_p|^{1+\f{1}{p}}
		\le t_p^{1+\f{1}{p}} \int_\Omega |\Delta \psi_p|^{1+\f{1}{p}}=\f{8\pi e}{p}(1+o_p(1)),$$
so \eqref{4-1} holds.
\ep

\bp[Proof of Corollary \ref{thm2}]
	Applying Theorem \ref{thm1} with $\beta=8\pi e$, we can prove the conclusion.
\ep

\vskip0.26in
\begin{appendices}

\end{appendices}




\vskip0.26in
\begin{thebibliography}{10}

\bibitem{asy1-3}
Adimurthi; Grossi, Massimo.
Asymptotic estimates for a two-dimensional problem with polynomial nonlinearity.
{\it Proc. Amer. Math. Soc.}, {\bf 132}(2004), no. 4, 1013-1019.

\bibitem{LiYY=1995}
Bahri, Abbas; Li, Yanyan; Rey, Olivier.
On a variational problem with lack of compactness: the topological effect of the critical points at infinity.
{\it Calc. Var. Partial Differential Equations}, {\bf 3}(1995), no. 1, 67-93.

\bibitem{exist3}
Bonheure, Denis; Moreira dos Santos, Ederson; Ramos, Miguel.
Ground state and non-ground state solutions of some strongly coupled elliptic systems.
{\it Trans. Amer. Math. Soc.}, {\bf 364}(2012), no. 1, 447-491.

\bibitem{1-estimate}
Br\'ezis, H\"aim; Merle, Frank.
Uniform estimates and blow-up behavior for solutions of $-\Delta u=V(x)e^u$ in two dimensions.
{\it Comm. Partial Differential Equations}, {\bf 16}(1991), no. 8-9, 1223-1253.

\bibitem{p-estimate}
Br\'ezis, H\"aim; Strauss, Walter A.
Semi-linear second-order elliptic equations in $L^1$.
{\it J. Math. Soc. Japan}, {\bf 25}(1973), 565-590.

\bibitem{classification}
Chanillo, S.; Kiessling, M. K.-H.
Conformally invariant systems of nonlinear PDE of Liouville type.
{\it Geom. Funct. Anal.}, 5 (1995), no. 6, 924-947.

\bibitem{asy3-2}
Choi, Woocheol; Kim, Seunghyeok.
Asymptotic behavior of least energy solutions to the Lane-Emden system near the critical hyperbola.
{\it J. Math. Pures Appl.}, {\bf (9)}132 (2019), 398-456.

\bibitem{Kim=2016}
Choi, Woocheol; Kim, Seunghyeok; Lee, Ki-Ahm.
Qualitative properties of multi-bubble solutions for nonlinear elliptic equations involving critical exponents.
{\it Adv. Math.}, {\bf 298}(2016), 484-533.

\bibitem{exist6}
Cl\'ement, Ph.; de Figueiredo, D. G.; Mitidieri, E.
Positive solutions of semilinear elliptic systems.
{\it Comm. Partial Differential Equations}, {\bf 17}(1992), no. 5-6, 923-940.

\bibitem{review1}
de Figueiredo, Djairo G.
Semilinear elliptic systems: existence, multiplicity, symmetry of solutions.
Handbook of differential equations: stationary partial differential equations.  1-48,
{\it Handb. Differ. Equ.}, Elsevier/North-Holland, Amsterdam, 2008.

\bibitem{exist1}
de Figueiredo, Djairo G.; do \'O, Jo\~ao Marcos; Ruf, Bernhard.
Critical and subcritical elliptic systems in dimension two.
{\it Indiana Univ. Math. J.}, {\bf 53}(2004), no. 4, 1037-1054.

\bibitem{exist4}
de Figueiredo, Djairo G.; Felmer, Patricio L.
On superquadratic elliptic systems.
{\it Trans. Amer. Math. Soc.}, {\bf 343}(1994), no. 1, 99-116.

\bibitem{asy2-4}
De Marchis, F.; Grossi, M.; Ianni, I.; Pacella, F.
$L^\iy$-norm and energy quantization for the planar Lane-Emden problem with large exponent.
{\it Arch. Math.}, {\bf 111}(2018), no. 4, 421-429.

\bibitem{asy2-1}
De Marchis, Francesca; Ianni, Isabella; Pacella, Filomena.
Asymptotic analysis and sign-changing bubble towers for Lane-Emden problems.
{\it J. Eur. Math. Soc.}, {\bf 17}(2015), no. 8, 2037-2068.

\bibitem{asy2-2}
De Marchis, Francesca; Ianni, Isabella; Pacella, Filomena.
Asymptotic analysis for the Lane-Emden problem in dimension two.
Partial differential equations arising from physics and geometry, 215-252,
{\it London Math. Soc. Lecture Note Ser.}, 450, Cambridge Univ. Press, Cambridge, 2019.

\bibitem{asy2-3}
De Marchis, Francesca; Ianni, Isabella; Pacella, Filomena.
Asymptotic profile of positive solutions of Lane-Emden problems in dimension two.
{\it J. Fixed Point Theory Appl.}, {\bf 19}(2017), no. 1, 889-916.

\bibitem{Druet-1}
Druet, O.
Multibumps analysis in dimension 2: quantification of blow-up levels.
{\it Duke Math. J.}, {\bf 132}(2006), no. 2, 217-269.

\bibitem{Druet-2}
Druet, Olivier; Thizy, Pierre-Damien.
Multi-bump analysis for Trudinger-Moser nonlinearities. I. Quantification and location of concentration points.
{\it J. Eur. Math. Soc.}, {\bf 22}(2020), no. 12, 4025-4096.

\bibitem{Druet-3}
Druet, O.; Robert, F.; Wei, J. C.
The Lin-Ni's problem for mean convex domains.
{\it Mem. Amer. Math. Soc.}, {\it 218}(2012), no. 1027, vi+105 pp.

\bibitem{book-2}
Gilbarg, David; Trudinger, Neil S.
Elliptic partial differential equations of second order.
Reprint of the 1998 edition. {\it Classics in Mathematics.} Springer-Verlag, Berlin, 2001.

\bibitem{asy1-4}
Grossi, Massimo; Grumiau, Christopher; Pacella, Filomena.
Lane-Emden problems: asymptotic behavior of low energy nodal solutions.
{\it Ann. Inst. H. Poincar\'e Anal. Non Lin\'eaire}, {\bf 30}(2013), no. 1, 121-140.

\bibitem{asy3-1}
Guerra, I. A.
Solutions of an elliptic system with a nearly critical exponent.
{\it Ann. Inst. H. Poincar\'e Anal. Non Lin\'eaire}, {\bf 25}(2008), no. 1, 181-200.

\bibitem{exist5}
Hulshof, Josephus; van der Vorst, Robertus.
Differential systems with strongly indefinite variational structure.
{\it J. Funct. Anal.}, {\bf 114}(1993), no. 1, 32-58.

\bibitem{Lions-1}
Lions, P. L.
The concentration-compactness principle in the calculus of variations. The locally compact case. I.
{\it Ann. Inst. H. Poincar\'e Anal. Non Lin\'eaire}, {\bf 1}(1984), no. 2, 109-145.

\bibitem{Lions-2}
Lions, P.-L.
The concentration-compactness principle in the calculus of variations. The locally compact case. II.
{\it Ann. Inst. H. Poincar\'e Anal. Non Lin\'eaire}, {\bf 1}(1984), no. 4, 223-283.

\bibitem{M1993}Mitidieri, Enzo. A Rellich type identity and applications: identity and applications. {\it Commun. Partial Differ. Equ.} {\bf 18}(1993), 125-151.

\bibitem{book-1}
Quittner, Pavol; Souplet, Philippe.
Superlinear parabolic problems. Blow-up, global existence and steady states.
Second edition.
{\it Birkh\"auser Advanced Texts: Basler Lehrbucher.}, Birkh\"{a}user/Springer, Cham, 2019.

\bibitem{asy1-1}
Ren, Xiaofeng; Wei, Juncheng.
On a two-dimensional elliptic problem with large exponent in nonlinearity.
{\it Trans. Amer. Math. Soc.}, {\bf 343}(1994), no. 2, 749-763.

\bibitem{asy1-2}
Ren, Xiaofeng; Wei, Juncheng.
Single-point condensation and least-energy solutions.
{\it Proc. Amer. Math. Soc.}, {\bf 124}(1996), no. 1, 111-120.

\bibitem{Struwe-1}
Struwe, Michael.
A global compactness result for elliptic boundary value problems involving limiting nonlinearities.
{\it Math. Z.}, {\bf 187}(1984), no. 4, 511-517.
\end{thebibliography}
 \end{document}